\documentclass[twoside,a4paper]{amsart}

\usepackage{amsmath,amssymb,amsthm,upgreek}
\usepackage{xcolor}
\definecolor{webgreen}{rgb}{0,.5,0}
\definecolor{webbrown}{rgb}{.6,0,0}
\definecolor{RoyalBlue}{cmyk}{1, 0.50, 0, 0}
\usepackage[colorlinks=true, breaklinks=true, urlcolor=webbrown, linkcolor=RoyalBlue, citecolor=webgreen,backref=page]{hyperref}
\usepackage{epsfig, graphicx,subfigure}
\usepackage{verbatim,setspace}

\newcommand{\R}		{\mathbb{R}}
\newcommand{\C}		{\mathbb{C}}
\newcommand{\N}		{\mathbb{N}}
\newcommand{\Z}		{\mathbb{Z}}
\newcommand{\cws}{\stackrel{*}{\to}}

\newcommand{\dd}{\mathrm{d}}
\newcommand{\ic}{\mathrm{i}}

\newcommand{\diam}{\mathsf{diam}}
\newcommand{\supp}{\mathsf{supp}}
\newcommand{\im}{\mathsf{Im}}
\newcommand{\re}{\mathsf{Re}}
\newcommand{\cp}{\mathsf{cp}}
\renewcommand{\deg}{\mathsf{deg}}
\newcommand{\qasq}{\quad\text{as}\quad}
\newcommand{\qandq}{\quad\text{and}\quad}
\newcommand{\qforq}{\quad\text{for}\quad}
\newcommand{\RS}{\boldsymbol{\mathfrak{R}}}
\newcommand{\bd}{\boldsymbol{\Delta}}
\newcommand{\mdp}{\mathrm{mod~periods~}}
\newcommand{\z}	{{\boldsymbol z}}
\newcommand{\x}	{{\boldsymbol x}}
\newcommand{\y}	{{\boldsymbol y}}
\newcommand{\e}	{{\boldsymbol e}}
\newcommand{\s}	{{\boldsymbol s}}
\newcommand{\ve}	{{\boldsymbol v}}
\newcommand{\ualpha}{\boldsymbol\upalpha}
\newcommand{\ubeta}{\boldsymbol\upbeta}
\newcommand{\ugamma}{\boldsymbol\upgamma}
\newcommand{\am}{\mathfrak{a}}
\newcommand{\rhy}   {\textnormal{RHP}-${\boldsymbol Y}$}
\newcommand{\rha}   {\textnormal{RHP}-\( \boldsymbol A \)}
\newcommand{\rhx}   {\textnormal{RH\(\bar\partial\)P}-\( \boldsymbol X \)}
\newcommand{\rhn}   {\textnormal{RHP}-${\boldsymbol N}$}
\newcommand{\rhr}   {\textnormal{RHP}-${\boldsymbol Z}$}
\newcommand{\pbd}  {\textnormal{\(\bar\partial\)P}-\( \boldsymbol D \)}

\newtheorem{theorem}{Theorem}[section]
\newtheorem{proposition}[theorem]{Proposition}

\newtheorem{lemma}[theorem]{Lemma}
\newtheorem{condition}{Condition}[section]
\newtheorem{definition}{Definition}[section]

\begin{document}

\title[Symmetric Contours and Convergent Interpolation]{Symmetric Contours and Convergent Interpolation}

\author[M. Yattselev]{Maxim L. Yattselev}

\address{Department of Mathematical Sciences, Indiana University-Purdue University Indianapolis, 402~North Blackford Street, Indianapolis, IN 46202}

\email{\href{mailto:maxyatts@iupui.edu}{maxyatts@iupui.edu}}

\thanks{The research was supported in part by a grant from the Simons Foundation, CGM-354538.}



\maketitle

\begin{abstract}

The essence of Stahl-Gonchar-Rakhmanov theory of symmetric contours as applied to the multipoint Pad\'e approximants is the fact that given a germ of an algebraic function and a sequence of rational interpolants with free poles of the germ, if there exists a contour that is ``symmetric'' with respect to the interpolation scheme, does not separate the plane, and in the complement of which the germ has a single-valued continuation with non-identically zero jump across the contour, then the interpolants converge to that continuation in logarithmic capacity in the complement of the contour. The existence of such a contour is not guaranteed. In this work we do construct a class of pairs interpolation scheme/symmetric contour with the help of hyperelliptic Riemann surfaces (following the ideas of Nuttall \& Singh \cite{NutS77} and Baratchart \& the author \cite{BY09c}). We consider rational interpolants with free poles of Cauchy transforms of non-vanishing complex densities on such contours under mild smoothness assumptions on the density. We utilize \( \bar\partial \)-extension of the Riemann-Hilbert technique to obtain formulae of strong asymptotics for the error of interpolation.

\end{abstract}

\maketitle

\section{Introduction}
\label{sec:intro}

Rational approximation of analytic functions is a very classical subject with various applications in number theory \cite{Herm73, Siegel, Skor03}, numerical analysis \cite{IserlesNorsett, BrezinskiRedivoZaglia}, modeling and control of signals and systems \cite{Antoulas, CameronKudsiaMansour, B_CMFT99,Partington2}, quantum mechanics and quantum field perturbation theory \cite{Baker, Tj_PRA77}, and many others. The theoretical aspects of the theory include the very possibility of such an approximation \cite{Run85,Mer62,Vit66} as well as the rates of convergence of the approximants at regular points when the degree grows large \cite{Walsh, Gon78c,Par86, Pr93}. 

In this work we are interested in rational interpolants with free poles,  the so-called {\em multipoint Pad\'e approximants} \cite{BakerGravesMorris}. Those are rational functions of type\footnote{A rational function is said to be of type \( (m,n) \) if it can be written as the ratio of a polynomial of degree at most \( m \) and a polynomial of degree at most \( n \).} \( (m,n) \) that interpolate a given function at \( m+n+1 \) points, counting multiplicity.  The beauty of multipoint Pad\'e approximants lies in the simplicity of their construction and the connection to (non-Hermitian) orthogonal polynomials. More precisely, the approximated function always can be written as a Cauchy integral of a complex density on any curve separating the interpolation points from the singularities of the function. The denominators of the multipoint Pad\'e approximants then turn out to be orthogonal to all the polynomials of smaller degree with respect to this density divided by the polynomial whose zeroes are the finite interpolation points. This connection is the most fruitful when the curve can be collapsed into a contour that does not separate the plane (as in the case of functions with algebraic and logarithmic singularities only). In general, there are many choices for such a contour with no obvious geometrical reason to prefer one over the other. The identification of the ``proper contour'', the one that attracts almost all of the poles of the approximants, is a fundamental question in the theory of Pad\'e approximation.

For the case of classical diagonal Pad\'e approximants (all the interpolation points are at infinity and \( m=n \)) to functions with branchpoints this question was answered in a series of pathbreaking papers \cite{St85, St85b, St86} by Stahl, where the approximants were shown to converge in capacity on the complement of the system of arcs of minimal logarithmic capacity outside of which the function is analytic and single-valued. The extremal system of arcs, called a {\it symmetric contour} or an {\em \( S \)-contour}, is characterized by the equality of the one-sided normal derivatives of its equilibrium potential at every smooth point of the contour, and the above-mentioned convergence ultimately depends on a deep potential-theoretic analysis of the zeros of non-Hermitian orthogonal polynomials. Shortly after, this result was extended by Gonchar and Rakhmanov \cite{GRakh87} to multipoint Pad\'e approximants to Cauchy integrals of continuous quasi-everywhere non-vanishing functions over contours minimizing now some \emph{weighted} capacity, provided that the interpolation points asymptotically distribute like a measure whose potential is the logarithm of the weight, see Section~\ref{sec:SGR} for a more detailed description of Stahl-Gonchar-Rakhmanov theory. 

These works clearly show that the appropriate Cauchy integrals for Pad\'e approximation must be taken over \( S \)-contours symmetric with respect to the considered interpolation schemes, if such contours exist. This poses a natural inverse problem: given a system of arcs, say \( \Delta \), is there an interpolation scheme turning \( \Delta \) into an symmetric contour? This inverse problem was first considered by Baratchart and the author in \cite{BY09c} for the case of a single Jordan arc. Below we build on the ideas of \cite{BY09c} by exhibiting a class of contours that are symmetric with respect to appropriately constructed interpolation schemes, see Section~\ref{ssec:SC}, and then derive formulae of \emph{strong} asymptotics for the error of approximation by multipoint Pad\'e approximants to Cauchy integrals of smooth densities on these contours, see Section~\ref{ssec:AA}.

\section{Stahl-Gonchar-Rakhmanov Theory}
\label{sec:SGR}

Throughout this section we always assume that \( f \) is a function holomorphic in a neighborhood of the point at infinity. The \emph{\( n \)-th diagonal Pad\'e approximant to \( f \)} is a rational function \( [n/n]_f = p_n/q_n \) of type \( (n,n) \) such that
\[
q_n(z)f(z)-p_n(z) = \mathcal O\big(1/z^{n+1}\big)  \qasq z\to\infty.
\]
Such a pair of polynomials always exists, the polynomial \( q_n \) of minimal degree is always unique, is never identically zero, and uniquely determines \( p_n \), see the explanation after Definition~\ref{def:pade} further below.

Our starting point is the following observation: if \( f \) is a germ of an algebraic function, then the approximants cannot converge everywhere outside of the branch points of \( f \) as their limit in capacity must be single-valued. Two questions immediately arise from this observation: do the approximants converge and if they do, where? To give answers to these question let us introduce a notion of an admissible compact. A compact set \( K \) is called \emph{admissible} for \( f \) if \( \overline\C\setminus K \) is connected and \( f \) has a meromorphic and single-valued extension there. The following theorem summarizes one of the fundamental contributions of Herbert Stahl to complex approximation theory \cite{St85,St85b,St86,St97}.

\begin{theorem}[{\bf Stahl}]
\label{thm:stahl}
Assume that the function \( f \) has a meromorphic continuation along any arc originating at infinity that belongs to \( \C\setminus E_f \) for some compact set \( E_f \) with \( \cp(E_f)=0 \)\footnote{\( \cp(\cdot) \) stands for the logarithmic capacity \cite{Ransford}.} and there do exist points in \( \C\setminus E_f \) that possess distinct continuations. Then
\begin{itemize}
\item[(i)] there exists the unique admissible compact \( \Delta_f \) such that \( \cp(\Delta_f)\leq\cp(K) \) for any admissible compact \( K \) and \( \Delta_f\subseteq K \) for any admissible \( K \) satisfying \( \cp(\Delta_f)=\cp(K) \). Pad\'e approximants \( [n/n]_f \) converge to \( f \)  in logarithmic capacity in \( D_f:=\overline\C\setminus\Delta_f \). The domain \( D_f \) is optimal in the sense that the convergence does not hold in any other domain \( D \) such that \( D\setminus D_f\neq\varnothing \).
\item[(ii)] the compact \( \Delta_f \) can be decomposed as \( \Delta_f=E_0 \cup E_1 \cup \bigcup\Delta_j \), where \( E_0\subseteq E_f \), \( E_1 \) consists of isolated points to which \( f \)  has unrestricted continuations from the point at infinity leading to at least two distinct function elements, and \( \Delta_j \) are open analytic arcs.
\item[(iii)] the Green function for \( D_f \) with pole at infinity, say \(g_{D_f} \)\footnote{The function \( g_{D_f} \) is harmonic and positive in \( D_f\setminus\{\infty\} \), its boundary values on \( \Delta_f \) vanish everywhere with a possible exception of a set of zero logarithmic capacity, and \( g_{D_f}(z)-\log|z| \) is bounded as \( z\to\infty \).}, possesses the S-property:
\[
\frac{\partial g_{D_f}}{\partial \boldsymbol{n}_+} = \frac{\partial g_{D_f}}{\partial \boldsymbol{n}_-} \quad \text{on} \quad \bigcup \Delta_j,
\]
where \( \partial/\partial\boldsymbol{n}_\pm \) are the one-sided normal derivatives on \( \bigcup \Delta_j \). Define
\[
 h_{D_f}(z):=\partial_zg_{D_f}(z), \quad 2\partial_z:=\partial_x-\mathrm{i}\partial_y.
\]
The function  \( h_{D_f}^2 \) is holomorphic in \( \overline\C\setminus(E_0\cup E_1) \), has a zero of order \( 2 \) at infinity, and the arcs \( \Delta_j \) are orthogonal critical trajectories of the quadratic differential \( h_{D_f}^2(z)\mathrm{d}z^2 \).
\item[(iv)] Assume in addition that \( f \)  is a germ of an algebraic function (\( E_f \) is necessarily finite). For each point \( e\in E_0\cup E_1 \) denote by \( i(e) \) the bifurcation index of \( e \), that is, the number of different arcs \( \Delta_j \) incident with \( e \). Then
\[
h_{D_f}^2(z) = \prod_{e\in E_0\cup E_1}(z-e)^{i(e)-2}\prod_{e\in E_2}(z-e)^{2j(e)},
\]
where \( E_2 \) is the set of critical points of \( g_{D_f} \) with \( j(e) \) standing for the \emph{order} of \( e\in E_2 \), i.e., \( \partial_z^jg_{D_f}(e)=0 \) for \( j\in\{1,\ldots,j(e)\} \) and \( \partial_z^{j(e)+1}g_{D_f}(e)\neq0 \), see Figure~\ref{fig:1}.
\end{itemize}
\end{theorem}

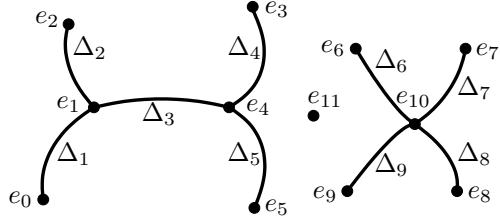
\begin{figure}[!ht]
\centering
\includegraphics[scale=.7]{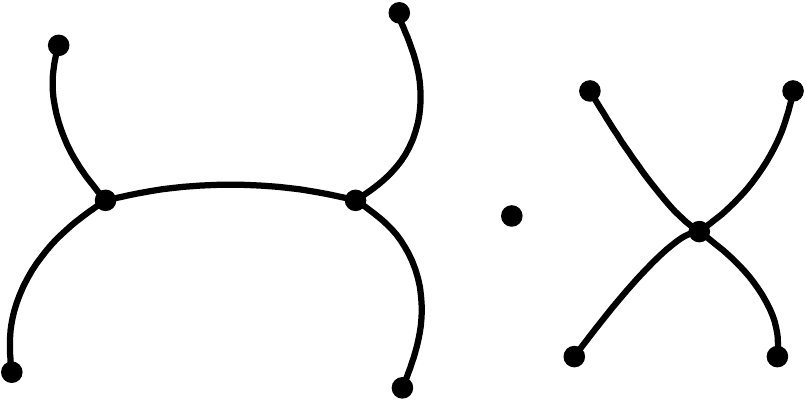}
\begin{picture}(0,0)
\put(-177,4){\( e_0 \)}
\put(-159,40){\( e_1 \)}
\put(-167,72){\( e_2 \)}
\put(-81,76){\( e_3 \)}
\put(-88,39){\( e_4 \)}
\put(-81,1){\( e_5 \)}
\put(-60,61){\( e_6 \)}
\put(-2,60){\( e_7 \)}
\put(-5,6){\( e_8 \)}
\put(-63,6){\( e_9 \)}
\put(-32,42){\( e_{10} \)}
\put(-65,42){\( e_{11} \)}
\put(-159,20){\( \Delta_1 \)}
\put(-153,60){\( \Delta_2 \)}
\put(-127,35){\( \Delta_3 \)}
\put(-95,60){\( \Delta_4 \)}
\put(-95,20){\( \Delta_5 \)}
\put(-40,55){\( \Delta_6 \)}
\put(-9,45){\( \Delta_7 \)}
\put(-9,20){\( \Delta_8 \)}
\put(-40,16){\( \Delta_9 \)}
\end{picture}
\caption{\small A possible shape of \( \Delta_f \). Generically, \( E_2 = \{ e_{11} \} \), \( E_1 =\{ e_1, e_4,e_{10}\} \), and the rest of the points belong to \( E_0 \). If some of the point \( e_1,e_4,e_{10} \) belong to \( E_f \), then they should be classified as elements of \( E_0 \). }
\label{fig:1}
\end{figure}

Classical Pad\'e approximants interpolate the function at one point, namely the point at infinity, with maximal order. However, one might want to interpolate it at several points. To this end, let \( V_{m+n}=\{v_{m+n,i}\}_{i=1}^{m+n} \) be a collection of not necessarily distinct nor finite points from a domain to which \( f \) possesses a single-valued holomorphic continuation.

\begin{definition}
\label{def:pade}
The multipoint Pad\'e approximant to \( f(z) \) associated with \( V_{m+n} \) of type \( (m,n) \) is a rational function \( [m/n;V_{m+n}]_f = p_{m,n}/q_{m,n} \) such that \( \deg(p_{m,n}) \leq m \), \( \deg(q_{m,n}) \leq n \), the linearized error
\begin{equation}
\label{Rn}
R_{m,n}(z) := \frac{q_{m,n}(z)f(z)-p_{m,n}(z)}{v_{m+n}(z)} = \mathcal O\left(z^{-\min\{m,n\}-1}\right) \qasq z\to\infty
\end{equation}
and has the same region of analyticity as \( f \), where \( v_{m+n} \) is the polynomial vanishing at finite elements of \( V_{m+n} \) according to their multiplicity\footnote{This definition yields linearized interpolation at the elements of \( V_{m+n} \) with one additional condition at infinity.}. We shall call the approximant diagonal if \( m=n \). Clearly, we recover the definition of the classical diagonal Pad\'e approximant when \( V_{2n} \) consists only of the points at infinity.
\end{definition}

The approximant \( [m/n;V_{m+n}]_f \) always exists as the conditions placed on \( R_{m,n} \) amount to solving a system of \( m+n+1 \) equations with \( m+n+2 \) unknowns. Observe that given the denominator polynomial, the numerator one is uniquely defined. Indeed, if \( p_1 \) and \( p_2 \) were to correspond to the same denominator, the expression \( (p_1 - p_2)/v_{m+n} \) would vanish at infinity with order at least \( \min\{m,n\}+1 \) and also at every zero of \( v_{m+n} \), which is clearly impossible. Moreover, one can immediately see from \eqref{Rn} that if \( p_1,q_1 \) and \( p_2,q_2 \) are solutions, then so is any linear combination \( c_1p_1 + c_2p_2,c_1q_1 + c_2q_2 \). Therefore, the solution corresponding to the monic denominator of the smallest degree is unique. In what follows, we understand that \( q_{m,n}, p_{m,n}, R_{m,n} \) come from this unique solution.

The most general result concerning the convergence in capacity of the diagonal multipoint Pad\'e approximants follows from the work of Gonchar and Rakhmanov \cite{GRakh87}. It deals more generally with the asymptotics of polynomials satisfying certain weighted non-Hermitian orthogonality relations of which denominators of the multipoint Pad\'e approximants are a particular example. Below we shall adduce their result solely within the framework of multipoint Pad\'e approximation. The starting point for \cite{GRakh87} is the generalization of the S-property introduced by Stahl.

\begin{definition}
\label{def:sym-GR}
Let $\Delta$ be a system of finitely many Jordan arcs that does not separate the plane and set $D:=\overline\C\setminus\Delta$. Assume that almost every point of $\Delta$ belongs to an analytic subarc. It is said that $\Delta$ is symmetric with respect to a positive Borel measure $\nu$ supported in $D$ (has the S-property w.r.t. \(\nu\)) if
\[
\frac{\partial g_D^\nu}{\partial \boldsymbol{n}_+} = \frac{\partial g_D^\nu}{\partial \boldsymbol{n}_-} \quad \text{a.e. on} \quad \Delta,
\]
where $\partial/\partial\boldsymbol{n}_\pm$ are the one-sided normal derivatives on \(\Delta\), $g_D^\nu(z):=\int g_D(z,u)\dd \nu(u)$ is the Green potential of $\nu$, and $g_D(\cdot,u)$ is the Green function for $D$ with pole at $u\in D$\footnote{When \(|u|<\infty\), \(g_D(z,u)\) is harmonic and positive in \(D\setminus\{u\}\), its boundary values on \(\Delta\) vanish everywhere with a possible exception of a set of zero logarithmic capacity, and \(g_D(z,u)+\log|z-u|\) is bounded as \(z\to u\).}.
\end{definition}

As the next step we choose an interpolation scheme that asymptotically approaches the measure \( \nu \). More precisely, given a function \( f \) and a collection of interpolations sets \( \mathcal V=\{V_{2n}\}_{n=1}^\infty \), we assume that
\begin{equation}
\label{asymp-measure}
\nu_n\cws\nu \qasq n\to\infty, \quad \nu_n := \frac1{2n}\sum_{i=1}^{2n}\delta(v_{2n,i}),
\end{equation}
where \(\delta(z)\) is the Dirac's delta distribution supported at \(z\)\footnote{The weak\( ^* \) convergence in the case of unbounded sets \( V_{2n} \) should be understood as follows: for any point \( a\notin \supp(\nu)\cup\bigcup_n V_{2n} \), the images of \( \nu_n \) under the map \( 1/(z-a) \) converge weak\( ^* \) to the image of \( \nu \) under the same map.}.

\begin{theorem}[{\bf Gonchar-Rakhmanov}]
Let \( \Delta \) be symmetric with respect to a positive Borel measure \( \nu \) supported in \( D=\overline\C\setminus\Delta \). If the function \( f(z) \) admits holomorphic continuation into \( D \) that we continue to denote by \( f \) and the jump of \( f \) across \( \Delta \) is non-zero almost everywhere, then the diagonal multipoint Pad\'e approximants \( [n/n;V_{2n}]_f \) associated with an interpolation scheme \( \mathcal V=\{V_{2n}\} \) asymptotic to \( \nu \) converge to \( f \) in logarithmic capacity in \( D \).
\end{theorem}

Observe that the above theorem assumes existence of a contour with an S-property while Stahl's theorem proves it but in a very specific case. Elaborating on Stahl's approach, Baratchart, Stahl, and the author \cite{BStY12} have shown that if the set \( E_f \) is finite and the measure \( \nu \) is supported outside of the smallest disk containing \( E_f \), then there exists a compact \( \Delta \) that is admissible for \( f \)  and is symmetric with respect to \( \nu \). Moreover, if \( E_f \) consists of two points, Baratchart and the author \cite{BY09c} proved that any Jordan arc connecting those points that is a conformal image of an interval is symmetric with respect to some measure supported in its complement. Finally, it is worth mentioning that in the framework of \cite{GRakh87}, sufficient conditions for existence of symmetric contours in harmonic fields were developed by Rakhmanov in \cite{Rakh12}. Let us stress that in \cite{Rakh12} given a harmonic field one looks for a system of arcs connecting certain points that is symmetric with respect to the field while in \cite{BY09c} and further below in Theorem~\ref{thm:scontours} one starts with a system of arcs for which a measure that makes it symmetric  is then produced (the corresponding field is given by the logarithmic potential of the measure).

\section{Main Results}
\label{sec:main}

This section is divided into four subsections. In the first one we adapt the definition of symmetry to our purposes (strong asymptotics) and state a result on existence of symmetric contours. In the second subsection we define all the functions necessary to describe asymptotics of the multipoint Pad\'e approximants, which is done in the third part of this section. Some numerical computations illustrating the theoretical results are presented in the final subsection.

\subsection{Symmetric Contours}
\label{ssec:SC}

Even before the work of Stahl, Nuttall and Singh \cite{NutS77} considered a class of contours that do satisfy Stahl's symmetry property, but were defined with the help of hyperelliptic Riemann surfaces. Below, we elaborate on this approach. To this end, let  \( E=\{e_0,\ldots,e_{2g+1}\} \) be a set of \( 2g+2 \) distinct points in \( \C \) and
\begin{equation}
\label{R}
\RS := \left\{ (z,w):~w^2 = \prod_{e\in E}(z-e),~z\in\overline\C \right\}
\end{equation}
be a hyperelliptic Riemann surface, necessarily of genus \( g \). Define the natural projection \( \pi:\RS\to\overline\C \) by \( \pi(z,w)=z \). We shall use bold lower case letters \( \z \), \( \s \), etc. to denote points on \( \RS \) with natural projections \( z \), \( s \), etc. We utilize the symbol \( \cdot^* \) for the conformal involution on \( \RS \), that is, \( \z^*=(z,-w) \) if \( \z=(z,w) \). Clearly, the set of ramification points of \( \RS \), namely \( \boldsymbol E=\{\e_0,\ldots,\e_{2g+1}\} \), is invariant under \( \cdot^* \).

\begin{definition}
\label{def:g-fun}
Given \( \ve \in \RS\setminus\boldsymbol E \), denote by \( g(\cdot,\ve) \) a function that is harmonic in \( \RS\setminus\{ \ve,\ve^*\} \), normalized so that \( g(\e_0,\ve) = 0 \), and such that
\[
g(\z,\ve) + \left\{\begin{array}{rl}
\log|z-v|, & |v|<\infty, \smallskip \\
-\log|z|, & v=\infty,
\end{array}
\right.
\qandq g(\z,\ve) - \left\{\begin{array}{rl}
\log|z-v|, & |v|<\infty, \smallskip \\
-\log|z|, & v=\infty,
\end{array}
\right.
\]
are harmonic as functions of \( \z \) around \( \ve \) and \( \ve^* \), respectively. For completeness, put \( g(\cdot,\e)\equiv 0 \) for \( \e\in\boldsymbol E \).
\end{definition}

Such a function always exists as it is simply the real part of an integral of the third kind differential with poles at \( \ve \) and \( \ve^* \) that have residues \( -1 \) and \( 1 \), respectively, and whose periods are purely imaginary. It readily follows from the maximum principle for harmonic functions that
\begin{equation}
\label{g-symmetry}
g(\z,\ve) + g(\z^*,\ve) = g(\z,\ve) + g(\z,\ve^*) \equiv 0 \qforq \z,\ve\in\RS.
\end{equation}

In what follows, we designate the symbol \( \mathcal V \) to stand for an interpolation scheme
\begin{equation}
\label{scheme}
\mathcal V = \big\{ V_{2n} \big\}_{n=1}^\infty, \quad V_{2n} =\big\{ v_{2n,i} \big\}_{i=1}^{2n}.
\end{equation}
Given \( v\in\overline\C\setminus E \), it will also be convenient to denote by \( \mathcal V_v \) the interpolation scheme that consists only of points \( v \).
The following definition is an extension of the one given in \cite{NutS77} to general interpolation schemes and the one given in \cite{BY09c} to the case \( g>0 \).

\begin{definition}
\label{def:symmetry}
Let \( \Delta \) be a system of open analytic arcs together with the set \( E \) of their endpoints and \( \mathcal V \) be an interpolation scheme in \( D:=\overline\C\setminus\Delta \). Further, let \( \RS \) be given by \eqref{R}. We say that \( \Delta \) is symmetric with respect to \( (\RS,\mathcal V) \) if
\begin{itemize}
\item[(i)] \( \RS\setminus\bd \), \( \bd :=\pi^{-1}(\Delta) \), consists of two disjoint connected open sets, say \( D^{(0)} \) and \( D^{(1)} \), and no closed subset of \( \Delta \) has this property; 
\item[(ii)] the sums \( \sum_{i=0}^{2n} g\big(\cdot,v_{2n,i}^{(0)}\big) \) are uniformly bounded above and below on \( \bd \) and go to \( -\infty \) locally uniformly in \( D^{(1)} \),  where \( v^{(i)}=\pi^{-1}(v)\cap D^{(i)} \) for \( v\in D \).
\end{itemize}
\end{definition}

The first condition in Definition~\ref{def:symmetry} says that \( \Delta \) does not separate the plane and can serve as a branch cut for \( w(z) \), see \eqref{R}, which has a non-zero jump across every subarc of \( \Delta \). The second one is essentially a non-Hermitian Blaschke-type condition.

To reconstruct the setting of \cite{NutS77}, put \( \mathcal V = \mathcal V_\infty \) in Definition~\ref{def:symmetry}. Then the second condition and \eqref{g-symmetry} imply that \( (-1)^ig\big(z^{(i)},\infty^{(0)}\big)>0 \) for \( z\in D \). Thus, by taking into account the first condition, we get that \( \bd :=\big\{\s: g\big(\s,\infty^{(0)}\big) = 0\big\} \). Consequently, we get that \( g\big(z^{(i)},\infty^{(0)}\big) = (-1)^ig_D(z) \), where \( g_D(z) \) is the Green function for \( D \) with pole at infinity. Therefore, the harmonic continuation of \( g_D(z) \) across each subarc of \( \Delta \) is given by \( -g_D(z) \). As we show later at the beginning of Section~\ref{sec:sym-proof}, this is equivalent to the S-property \( \partial g_D/\partial \boldsymbol n_+ = \partial g_D / \partial \boldsymbol n_- \) on \( \Delta \).

The connection between Definition~\ref{def:symmetry} and the notions of symmetry from Theorem~\ref{thm:stahl} and Definition~\ref{def:sym-GR} is rather straightforward.

\begin{proposition}
\label{prop:connections}
In the setting of Theorem~\ref{thm:stahl}(iv), set \( E \) to be the subset of \( E_0\cup E_1 \) comprised of all the points with odd bifurcation index, i.e., the branch points of \( h_{D_f} \). Then \( \Delta_f \) is symmetric with respect to \( (\RS ,\mathcal V_\infty) \) in the sense of Definition~\ref{def:symmetry}, where \( \RS \) is associated to \( E \) via \eqref{R}.

In another connection, let \( \Delta \) be symmetric with respect to \( (\RS,\mathcal V) \) in the sense of Definition~\ref{def:symmetry}. Assume that the interpolation scheme \( \mathcal V \) is separated from \( \Delta \) and asymptotically approaches the measure \( \nu \), \( \supp(\nu)\subset D=\overline\C\setminus\Delta \), in the sense of \eqref{asymp-measure}. Then \( \Delta \) is symmetric with respect to \( \nu \) in the sense of Definition~\ref{def:sym-GR}. 
\end{proposition}

Concerning the existence of symmetric contours, we can say the following. 

\begin{theorem}
\label{thm:scontours}
Given \( \RS \) as in \eqref{R} and \( v\in\overline\C\setminus E \), there always exists a contour \( \Delta_v \) symmetric with respect to \( (\RS,\mathcal V_v) \). Further, let \( c>0 \) be a constant such that \( L_c :=\{ s: g_{D_v}(s) = c \} \) is a smooth Jordan curve, where \( D_v:=\overline\C\setminus\Delta_v \). If \( \Xi(z) \) is a conformal function in the interior of \( L_c \) such that \( \Xi(e) =e \) for every \( e\in E \), then there exists an interpolation scheme \( \mathcal V \) in \( \overline\C\setminus\Xi(\Delta) \) such that \( \Xi(\Delta) \) is symmetric with respect to \( (\RS,\mathcal V) \).
\end{theorem}

We prove Proposition~\ref{prop:connections} and Theorem~\ref{thm:scontours} in Section~\ref{sec:sym-proof}.

\subsection{Nuttall-Szeg\H{o} Functions}

Given \( \Delta \) as in Definition~\ref{def:symmetry}(i), we realize \( \RS \), the Riemann surface of \( w \), as 
\begin{equation}
\label{realization}
\RS = D^{(0)} \cup \bd \cup D^{(1)}, \quad \bd := \pi^{-1}(\Delta), \quad D^{(0)}\cup D^{(1)} := \pi^{-1}(D),
\end{equation}
where the open sets \( D^{(i)} \) are connected and \( \pi(D^{(i)}) = D \), \( i\in\{0,1\} \). For convenience we shall also denote by \( z^{(i)} \) the lift of \( z\in D \) to \( D^{(i)} \). We denote by \( \big\{\ualpha_k,\ubeta_k\big\}_{k=1}^g \) a homology basis\footnote{The surface cut along the cycles of a homology basis becomes simply connected, \( \ualpha_k,\ubeta_k \) intersect once and form the right pair at the point of intersection, different \( \ualpha \)-cycles do not intersect as well as different \( \ubeta \)-cycles.} on \( \RS \) from which we only require that each cycle is \emph{involution-symmetric} (i.e., \( \ugamma=\{\z^*~|~\z\in\ugamma\} \) ) and has only finitely many points in common with \( \bd \), see Figure~\ref{fig:2}.

\begin{figure}[!ht]
\centering
\includegraphics[scale=.7]{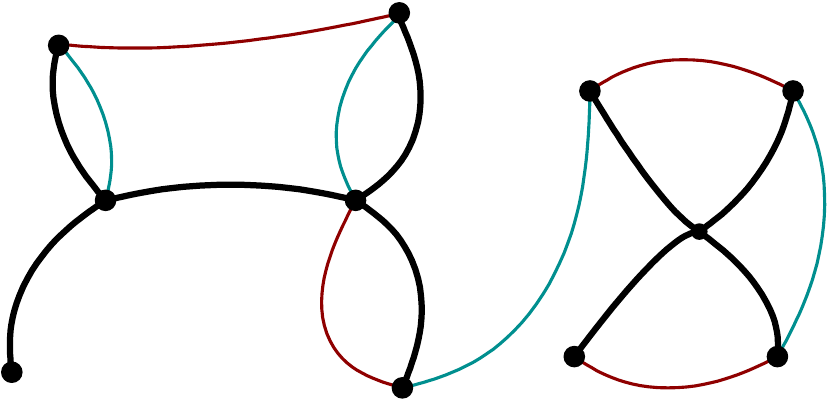}
\begin{picture}(0,0)
\put(-182,4){\( \e_0 \)}
\put(-164,40){\( \e_1 \)}
\put(-172,72){\( \e_2 \)}
\put(-86,76){\( \e_3 \)}
\put(-93,39){\( \e_4 \)}
\put(-86,-2){\( \e_5 \)}
\put(-65,61){\( \e_6 \)}
\put(-7,60){\( \e_7 \)}
\put(-10,6){\( \e_8 \)}
\put(-68,6){\( \e_9 \)}
\put(-130,76){\( \ualpha_1 \)}
\put(-120,20){\( \ualpha_2 \)}
\put(-32,72){\( \ualpha_3 \)}
\put(-40,6){\( \ualpha_4 \)}
\put(-147,54){\( \ubeta_1 \)}
\put(-115,60){\( \ubeta_2 \)}
\put(-70,30){\( \ubeta_3 \)}
\put(0,35){\( \ubeta_4 \)}
\put(-200,70){\( D^{(0)} \)}
\end{picture}
\caption{\small The set of ramification points \( \boldsymbol E =\{\e_0,\ldots,\e_9 \} \) of \( \RS \) as seen from \( D^{(0)} \) and ``half '' of the cycle \( \bd \) and the cycles of the homology basis. Due to involution-symmetry, the picture in \( D^{(1)} \) is identical. The point where four subarcs of \( \Delta \) meet is not a ramification point of \( \RS \). }
\label{fig:2}
\end{figure}
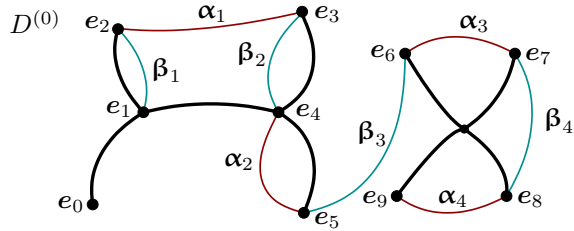

Denote by \( \vec\Omega:=\left(\Omega_1,\ldots,\Omega_g\right)^\mathsf{T} \) the column vector of \( g \) linearly independent holomorphic differentials\footnote{It holds that \( \Omega_k(\z) = (L_k/w)(\z)\dd z\), where \( L_k(z) \) is a certain polynomial of degree at most \( g-1 \) lifted to \( \RS \) and \( w(\z) := (-1)^iw(z) \) for \( \z\in D^{(i)} \).} normalized so that \( \oint_{\ualpha_k}\vec\Omega = \vec e_k \), where \( \{\vec e_k\}_{k=1}^g \) is the standard basis for \( \R^g \) and \( \vec e^\mathsf{T} \) is the transpose of \( \vec e \). Set
\begin{equation}
\label{B}
\boldsymbol B := \left[ \oint_{\ubeta_j}\Omega_k\right]_{j,k=1}^g.
\end{equation}
It is known that the matrix \( \boldsymbol B \) is symmetric and has positive definite imaginary part. 

A divisor on \( \RS \) is a finite linear combination of points from \( \RS \) with integer coefficients. The degree of a divisor is the sum of its coefficients. The divisor is called \emph{effective} if all the coefficients are non-negative. We define Abel's map on divisors of \( \RS \) by
\begin{equation}
\label{abel-map}
\am\left(\sum n_j\z_j\right) := \sum n_j\int_{\e_0}^{\z_j}\vec\Omega.
\end{equation}
A divisor \( \mathcal D = \sum n_j \z_j - \sum m_j\x_j\), \( n_j,m_j>0 \), is called \emph{principal} if there exists a rational function on \( \RS \) that has a zero at every \( \z_j \) of multiplicity \( n_j \), a pole at every \( \x_j \) of order \( m_j \), and otherwise is non-vanishing and finite. By Abel's theorem, \( \mathcal D \) is principal if and only if its degree is zero and
\[
\am(\mathcal D) \equiv \vec 0 \quad \big(\mdp\vec\Omega\:\big),
\]
where the equivalence of two vectors \( \vec c,\vec u\in \C^g \) is defined by \( \vec c\equiv \vec u \) \( \big(\mdp\vec\Omega\:\big) \) if and only if \( \vec c - \vec u = \vec j + \boldsymbol B\vec m \), for some \( \vec j,\vec m\in\Z^g \).

For any point \( \ve\in\RS\setminus\boldsymbol E \) there exists a unique differential, say \( G_\ve \), such that it is holomorphic on \( \RS\setminus\{\ve,\ve^*\} \), has polar singularities at \( \ve \) and \( \ve^* \) with respective residues \( -1 \) and \( 1 \), and whose periods are purely imaginary. Given \( \mathcal V \) as in \eqref{scheme}, define vectors \( \vec\omega_n \) and \( \vec \tau_n \) by
\begin{equation}
\label{constants}
(\vec\omega_n)_k:=-\frac1{4\pi\mathrm{i}}\sum_{i=1}^{2n}\oint_{\ubeta_k}G_{v_{2n,i}^{(0)}} \qandq (\vec\tau_n)_k:=\frac1{4\pi\mathrm{i}}\sum_{i=1}^{2n}\oint_{\ualpha_k}G_{v_{2n,i}^{(0)}},
\end{equation}
where we adopt the notation \( (\vec c)_k := c_k \) for \( \vec c=(c_1,\ldots,c_g) \). Notice that these constants are real. Given a continuous function \( \rho \) on \( \Delta \), we are interested in the solutions of the following Jacobi inversion problem: find an effective divisor \( \mathcal D_n \) of degree \( g \) such that
\begin{equation}
\label{main-jip}
\am(\mathcal D_n)\, \equiv \, \am\big(g\infty^{(1)}\big) + \vec c_\rho + \vec\omega_n + \boldsymbol B\vec\tau_n,  \quad \big(\mdp \vec\Omega\:\big),
\end{equation}
where \( \vec c_\rho := \frac1{2\pi\mathrm{i}}\oint_{\bd}\log(\rho\circ\pi)\vec\Omega \). This problem is always solvable and the solution is unique up to a principal divisor. That is, if \( \mathcal D_n - \big\{\mbox{ principal divisor }\big\} \) is an effective divisor, then it also solves \eqref{main-jip}. Immediately one can see that the subtracted principal divisor should have a positive part of degree at most \( g \). As \( \RS \) is hyperelliptic, such divisors come solely from rational functions on \( \overline\C \) lifted to \( \RS \). In particular, such principal divisors are involution-symmetric. Hence, if a solution of \eqref{main-jip}  contains at least one involution-symmetric pair of points, then replacing this pair by another such pair produces a different solution of \eqref{main-jip}. However, if a solution does not contain such a pair, then it solves \eqref{main-jip} uniquely. 

\begin{proposition}
\label{prop:nuttall-szego}
Let \( \rho \) be a H\"older continuous and non-vanishing function on \( \Delta \). If \eqref{main-jip} is uniquely solvable for a given index \( n \), then there exists a sectionally meromorphic in \( \RS\setminus\bd \) function \( \Psi_n \) whose zeros and poles there are described by the divisor
\begin{equation}
\label{PsiZeros}
(n-g)\infty^{(1)}+\mathcal{D}_n-n\infty^{(0)},
\end{equation}
and whose traces on \( \bd \) are continuous and satisfy
\begin{equation}
\label{PsiBoundary}
\Psi_{n-}(\x) = \big(\rho(x)/v_{2n}(x)\big) \Psi_{n+}(\x), \quad \x\in\bd.
\end{equation}
Moreover, if \( \Psi \) is a sectionally meromorphic function in \( \RS\setminus\bd \)  satisfying \eqref{PsiBoundary} whose divisor has a form \( 
(n-g)\infty^{(1)}+\mathcal D -n\infty^{(0)} \) for some effective divisor \( \mathcal D \), then \( \Psi \) is a constant multiple of~\( \Psi_n \). 
\end{proposition}

Together with \( \Psi_n \) we shall need the following sequence of functions. 

\begin{proposition}
\label{prop:secondary}
Let an index \( n \) be such that \eqref{main-jip} is uniquely solvable. If \( \mathcal D_n \) does not contain \( \infty^{(0)} \), then there exists a unique, up to a constant factor, rational function \( \Upsilon_n \) on \( \RS \) such that
\[ 
(\Upsilon_n) + \mathcal D_n + \infty^{(1)} - \infty^{(0)}
\]
is an effective divisor, where \( (\Upsilon_n) \) is the divisor of the zeros and poles of \( \Upsilon_n \). Moreover, in this case \( \Upsilon_n \) always has a simple pole at \( \infty^{(1)} \).
\end{proposition}

Effective divisors of degree \( g \) can be considered as elements of \( \RS^g/\Sigma_g \), the quotient of \( \RS^g \) by the symmetric group \(\Sigma_g \), which is a compact topological space. Thus, it make sense to talk about the limit points of \( \{\mathcal{D}_n\} \). We shall assume that

\begin{condition}
\label{CondE}
There exists an infinite sequence $\N_*\subseteq\N$ such that the closure of \( \big\{\mathcal{D}_n\big\}_{n\in\N_*} \) in the \( \RS^g/\Sigma_g\)-topology contains no divisor with an involution-symmetric pair nor with \( \infty^{(0)} \).
\end{condition}

Observe that \eqref{main-jip} is necessarily uniquely solvable for every \( n\in \N_* \). 

\begin{proposition}
\label{prop:estimates}
Assume Condition~\ref{CondE} is satisfied. Then the functions \( \Psi_n \) and \( \Upsilon_n \) can be normalized so that
\begin{equation}
\label{maxPsi}
\big|\Psi_n\big(z^{(1)}\big)\big|^2,\big|(\Psi_n\Upsilon_n)\big(z^{(1)}\big)\big|^2 \leq C\exp\left\{ \sum_{i=0}^{2n} g\big(z^{(1)},v_{2n,i}^{(0)}\big) \right\}\left|\frac{w^2(z)}{v_{2n}(z)}\right|,
\end{equation}
\( \quad n\in\N_* \), on any closed set \( K\subset D \) for some constant \( C=C(K,\N_*) \).
\end{proposition}

Recall that according to Definition~\ref{def:symmetry}(ii) the exponential in the right-hand side of \eqref{maxPsi} vanishes at every zero of \( v_{2n}(z) \) with corresponding multiplicity and their sequence approaches zero locally uniformly in \( D \).

Concerning the unique solvability of \eqref{main-jip} and the existence of a sequence \( \N_* \) as in Condition~\ref{CondE} nothing is known beyond the special case of the classical diagonal Pad\'e approximants, i.e., when \( \mathcal V=\mathcal V_\infty \) \cite{ApY15,Y15}.

\begin{theorem}[{\bf Aptekarev-Y.}]
Assume that \( \mathcal V=\mathcal V_\infty \). Let \( \mathcal D_n \) be either the unique solution of \eqref{main-jip} or the solution where all involution-symmetric pairs are replaced by \( \infty^{(0)} + \infty^{(1)} \). Then
\[
\mathcal{D}_n=\sum_{i=1}^{g-l}\z_i+k\infty^{(0)}+(l-k)\infty^{(1)} \quad \Leftrightarrow \quad \mathcal{D}_{n+j} = \mathcal{D}_n +j\big(\infty^{(0)}-\infty^{(1)}\big)
\]
for \( j\in\{-k,\ldots,l-k\} \), where \( l>0 \), \( k\in\{0,\ldots,l\}$, and \( |z_i|<\infty \). In particular, the subsequence of indices for which \eqref{main-jip} is uniquely solvable, say \( \N^\prime \), cannot have gaps larger than \( g-1 \). Moreover, let \( \N^{\prime\prime}\subset\N^\prime \) be a subsequence such that
\[
\mathcal D_n \to \mathcal D + \sum_{i=1}^k\left(z_i^{(0)}+z_i^{(1)}\right) + l_0\infty^{(0)} + l_1\infty^{(1)} \qasq \N^{\prime\prime} \ni n\to\infty,
\]
where an effective divisor \( \mathcal D \) has degree \( g-2k-l_0-l_1 \) and contains neither involution-symmetric pairs, nor \( \infty^{(0)} \), nor \( \infty^{(1)} \). Then there exists a subsequence \( \N^{\prime\prime\prime} \) such that
\[
\mathcal D_{n-l_1-k} \to \mathcal D + (l_0 + l_1+2k )\infty^{(1)} \qasq \N^{\prime\prime\prime} \ni n\to\infty.
\]
In particular, one can take \( \N_* = \N^{\prime\prime\prime} \).
\end{theorem}

We prove Propositions~\ref{prop:nuttall-szego}--\ref{prop:estimates} in Section~\ref{sec:NS}.

\subsection{Asymptotics of the Approximants}
\label{ssec:AA}

Given \( \Delta \) as in Definition~\ref{def:symmetry}(i) and a function \( \rho \) on \( \Delta \), set
\begin{equation}
\label{f-rho}
f_\rho(z) := \frac1{2\pi\mathrm{i}}\int_\Delta\frac{\rho(t)}{t-z}\frac{\dd t}{w_+(t)}, \quad z\in D.
\end{equation}
We shall be interested in continuous and non-vanishing functions \( \rho \) such that a continuous determination of the logarithm \( \log\rho \) belongs to the fractional Sobolev space \( W^{1-1/p}_p \), \( p\in(1,\infty) \), that is,
\begin{equation}
\label{sobolev}
\iint_{\Delta\times\Delta}\left|\frac{\log\rho(x)-\log\rho(y)}{x-y}\right|^p|\dd x||\dd y|<\infty.
\end{equation}
When \( p>2 \), it follows from Sobolev imbedding theorem that every function in \( W^{1-1/p}_p \) is in fact Lipschitz continuous with index \( 1-\frac2p \). For convenience, we also put
\begin{equation}
\label{Psis}
\left\{
\begin{array}{l}
\Psi_n(z) := \Psi_n\big(z^{(0)}\big), \medskip \\
\Psi_n^*(z) := \Psi_n\big(z^{(1)}\big),
\end{array}
\right. \qandq
\left\{
\begin{array}{l}
\Upsilon_n(z) := \Upsilon_n\big(z^{(0)}\big), \medskip \\
\Upsilon_n^*(z) := \Upsilon_n\big(z^{(1)}\big),
\end{array}
\right.
\end{equation}
\( z\in D \), where \( \Psi_n, \Upsilon_n \) are the functions from Propositions~\ref{prop:nuttall-szego}--\ref{prop:estimates}. Then the following theorem holds.

\begin{theorem}
\label{thm:main}
Given \( \RS \) and \( \mathcal V \) as in \eqref{R} and \eqref{scheme}, assume that \( \Delta \) is symmetric with respect to \( (\RS,\mathcal V) \) in the sense of Definition~\ref{def:symmetry}. Let \( \rho \) be a non-vanishing function on \( \Delta \) with \( \log\rho\in W_p^{1-1/p} \) for some \( p>4 \) and let \( f_\rho \) be as in \eqref{f-rho}. Further, let
\[
p_n/q_n:=[n/n;V_{2n}]_{f_\rho} \qandq R_n:=R_{n,n}
\]
be the diagonal multipoint Pad\'e approximant to \( f_\rho \) associated with \( V_{2n} \) and its linearized error function, see Definition~\ref{def:pade}. Assuming that the interpolation scheme \( \mathcal V \) is such that Condition~\ref{CondE} is fulfilled, it holds for all $n\in\N_*$ large enough that
\begin{equation}
\label{SA}
\left\{
\begin{array}{rll}
q_n &=& \displaystyle \gamma_n\Psi_n\left( 1+\varepsilon_{n1} + \varepsilon_{n2} \Upsilon_n\right), \bigskip \\
wR_n &=& \displaystyle \gamma_n\Psi_n^*\left(1+\varepsilon_{n1} +\varepsilon_{n2}\Upsilon_n^*\right),
\end{array}
\right.
\end{equation}
where \( \varepsilon_{nj}(\infty)=0 \), \( \varepsilon_{nj} =o(1) \) with respect to \( n\) and locally uniformly in \( D \), and \( \gamma_n \) is a normalizing constant such that \( \gamma_n\Psi_n(z) =z^n(1+o(1) ) \) as \( z\to\infty \). 
\end{theorem}

In the case of classical Pad\'e approximants Theorem~\ref{thm:main} should be compared to results by Szeg\H{o} \cite{Szego} (\( \Delta=[-1,1] \) and \( \rho(t)\dd t/w_+(t) \) is replaced by any positive measure satisfying Szeg\H{o}'s condition); Nuttall \cite{Nut90} (\( \Delta=[-1,1] \) and \( \rho \) is H\"older continuous); Suetin \cite{Suet00} (\( \Delta \) is a union of disjoint analytic arcs and \( \rho \) is H\"older continuous); Baratchart and the author \cite{BY13} (\( \Delta \) consists of three arcs meeting at a common point and \( \rho \) is Dini continuous); Aptekarev and the author \cite{ApY15} (\( \Delta \) is such that no endpoint has bifurcation index more than \( 3 \), \( \rho \) is holomorphic across each \( \Delta_j \) and can have power-type singularities at endpoints with bifurcation index \( 1\)); and the author \cite{Y15} (\( \Delta \) is any and \( \rho \) is holomorphic around each connected component of \( \Delta \)). Of course, in all the cases \( \Delta \) is a symmetric contour and \( \rho(t) \) is non-vanishing (except for Szeg\H{o}'s result).

In the case of multipoint Pad\'e approximants Theorem~\ref{thm:main} is an addition to the results by de la Calle Ysern and L\'opez Lagomasino \cite{CYLL98} (Szeg\H{o}'s set up with interpolation schemes as in the present study plus additional conjugate-symmetry); Baratchart and the author \cite{BY09c,BY10} (\( \Delta \) is a single arc and \( \rho \) is Dini-continuous in \cite{BY09c} and with power-type singularities at the endpoints while satisfying Sobolev-type condition that depends on the magnitude of the singularities on \( \Delta^\circ \) in \cite{BY10}, the class of interpolation schemes is more restricted in \cite{BY10} while in \cite{BY09c} they are exactly the same as in the present paper); Aptekarev \cite{Ap02} (it is a more general result on varying non-Hermitian orthogonality that can be applied to multipoint Pad\'e approximants to yield the results of \cite{BY09c,BY10}, which came later, for holomorphic \( \rho \)).

We prove Theorem~\ref{thm:main} in Section~\ref{sec:asymp}.

\subsection{Numerical Simulations}

The following computations were performed in MAPLE 18 software using 64 digit precision.

\begin{figure}[ht!]
\centering
\subfigure[]{\includegraphics[scale=.25]{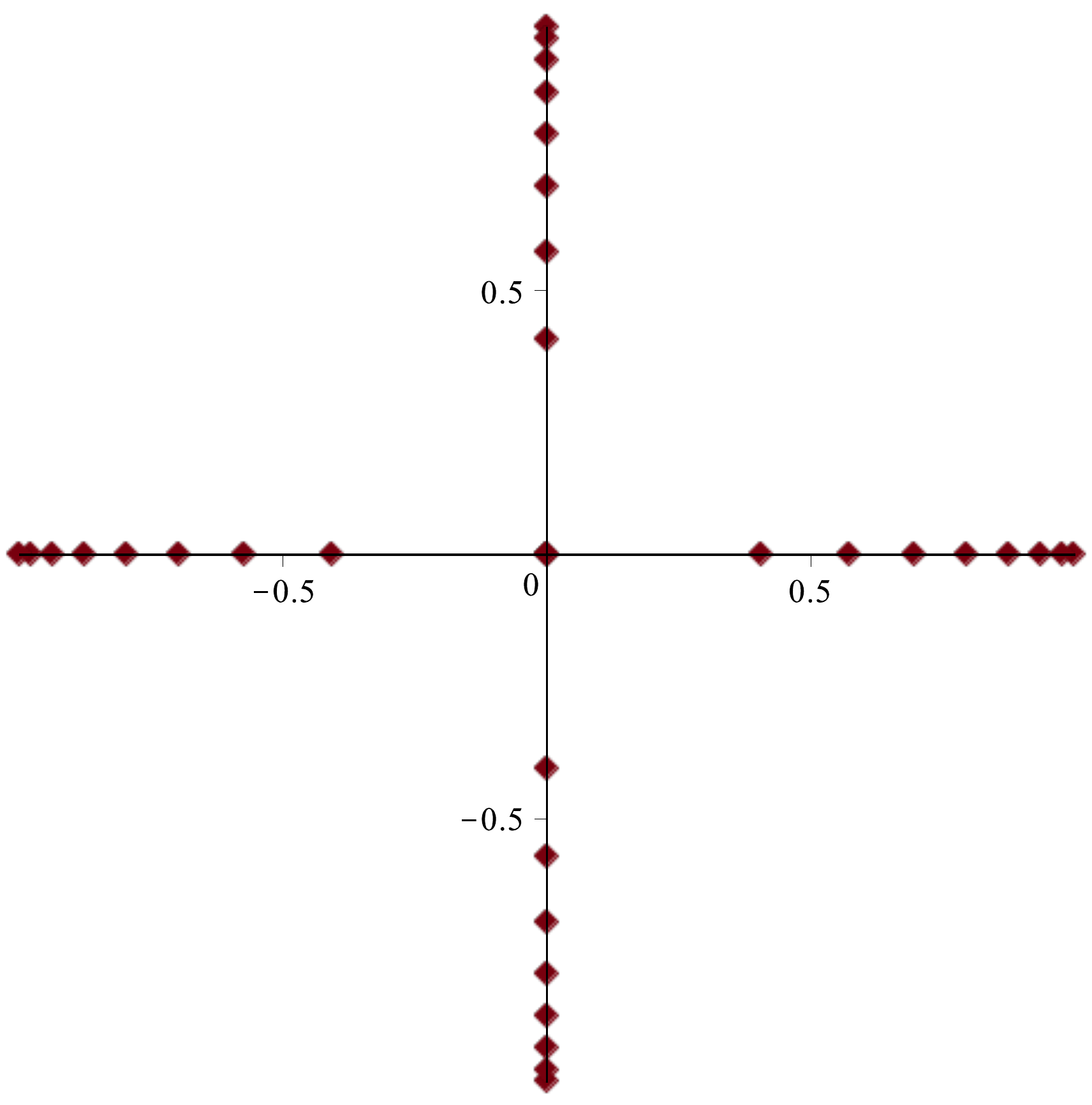}}~
\subfigure[]{\includegraphics[scale=.25]{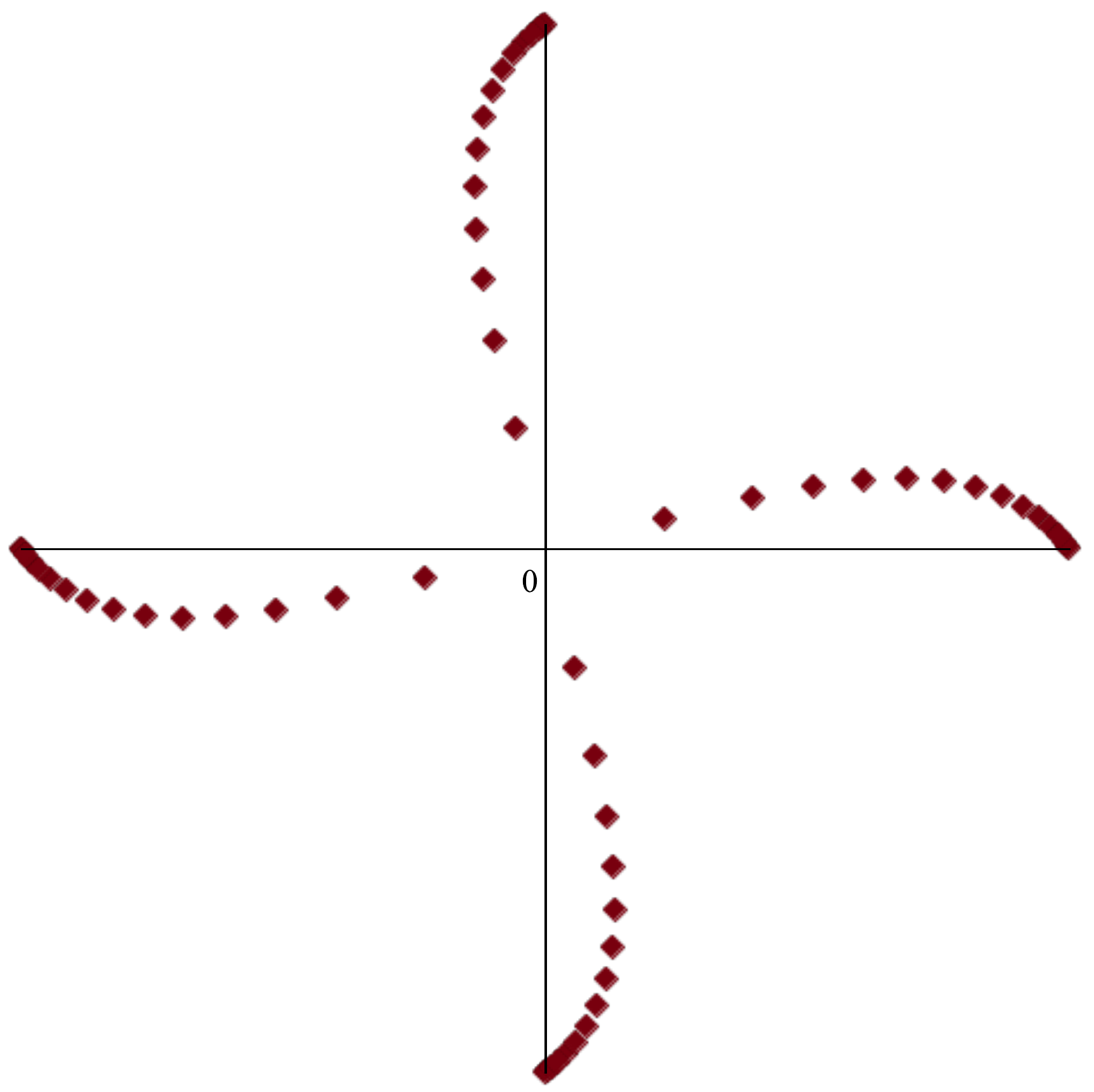}}~
\subfigure[]{\includegraphics[scale=.25]{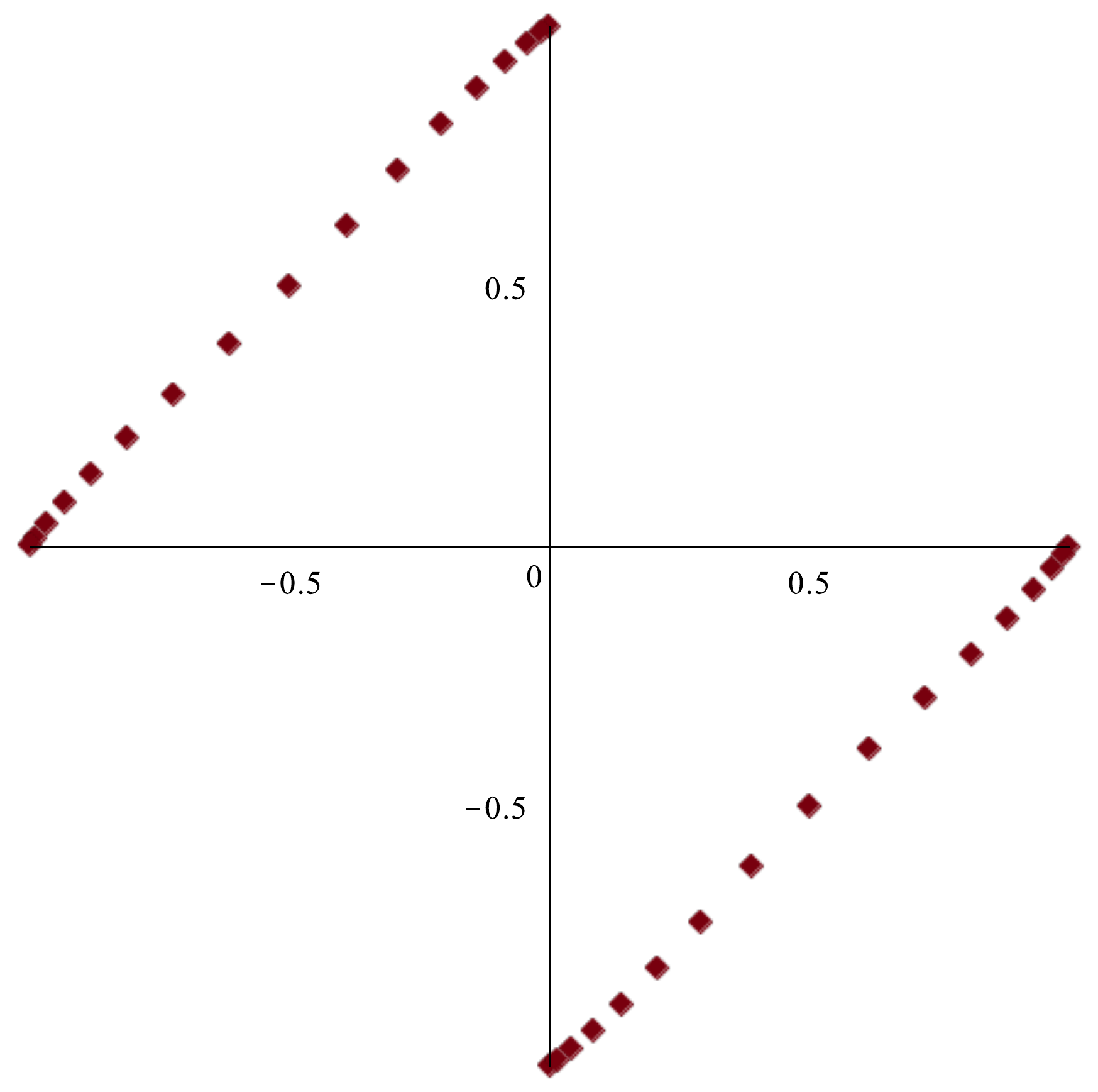}}
\caption{Poles of multipoint Pad\'e approximants to \( (z^4-1)^{-1/2} \) associated with interpolation schemes \( \mathcal V_* \), \( \mathcal V_*^\prime \), and \( \mathcal V_*^{\prime\prime} \).}
\label{fig:f2}
\end{figure}

Let \( f(z) = (z^4-1)^{-1/2} \). The symmetry of \( f(z) \) readily implies that the Stahl's contour \( \Delta_f=\Delta_\infty \) from Theorem~\ref{thm:stahl} is equal to \( [-1,1]\cup[-\ic,\ic] \). The corresponding surface from \eqref{R} is given by \( \RS_* := \big\{(z,w):~w^2=z^4-1\big\} \). Similar symmetry considerations also yield that for the interpolation scheme \( \mathcal V_* \) such that \( V_{4n+2} = V_{4n}\cup\{\infty,\infty\} \) and
\[
V_{4n} = \left\{\underbrace{1+\ic,\ldots,1+\ic}_{n~\text{times}},\underbrace{-1+\ic,\ldots,-1+\ic}_{n~\text{times}},\underbrace{-1-\ic,\ldots,-1-\ic}_{n~\text{times}},\underbrace{1-\ic,\ldots,1-\ic}_{n~\text{times}}\right\},
\]
\( [-1,1]\cup[-\ic,\ic] \) remains symmetric with respect to \( (\RS_*,\mathcal V_*) \). The poles of \( [34/34;V_{68}]_f \) are shown on Figure~\ref{fig:f2}(a). If we take now \( V_{4n+2}^\prime = V_{4n}^\prime\cup\{\infty,\infty\} \) and
\[
V_{4n}^\prime = \left\{\underbrace{\frac14+\ic,\ldots,\frac14+\ic}_{n~\text{times}},\underbrace{-1+\frac\ic4,\ldots,-1+\frac\ic4}_{n~\text{times}},\underbrace{-\frac14-\ic,\ldots,-\frac14-\ic}_{n~\text{times}},\underbrace{1-\frac\ic4,\ldots,1-\frac\ic4}_{n~\text{times}}\right\},
\]
then the poles of \( [60/60;V_{120}^\prime]_f \) are shown on Figure~\ref{fig:f2}(b). The Riemann surfaces needed to analyze these approximants is still \( \RS_* \) and the contour symmetric with respect to \( (\RS,\mathcal V_*^\prime) \) can be obtained via the process described in Theorem~\ref{thm:scontours}. If we take
\[
V_{2n}^{\prime\prime}=\left\{\underbrace{1+\ic,\ldots,1+\ic}_{n~\text{times}},\underbrace{-1-\ic,\ldots,-1-\ic}_{n~\text{times}}\right\},
\]
then the poles of \( [34/34;V_{68}^{\prime\prime}]_f \) are shown on Figure~\ref{fig:f2}(c). As in the previous case, \( \RS_* \) is still the appropriate Riemann surface, but the contour symmetric with respect to \( (\RS_*,\mathcal V_*^{\prime\prime}) \) cannot be obtained via the procedure of Theorem~\ref{thm:scontours}. It is most likely that a version of Theorem~\ref{thm:scontours} where the map \( \Xi \) is defined on the surface itself, could prove the existence of such a contour.

\begin{figure}[ht!]
\centering
\subfigure[]{\includegraphics[scale=.25]{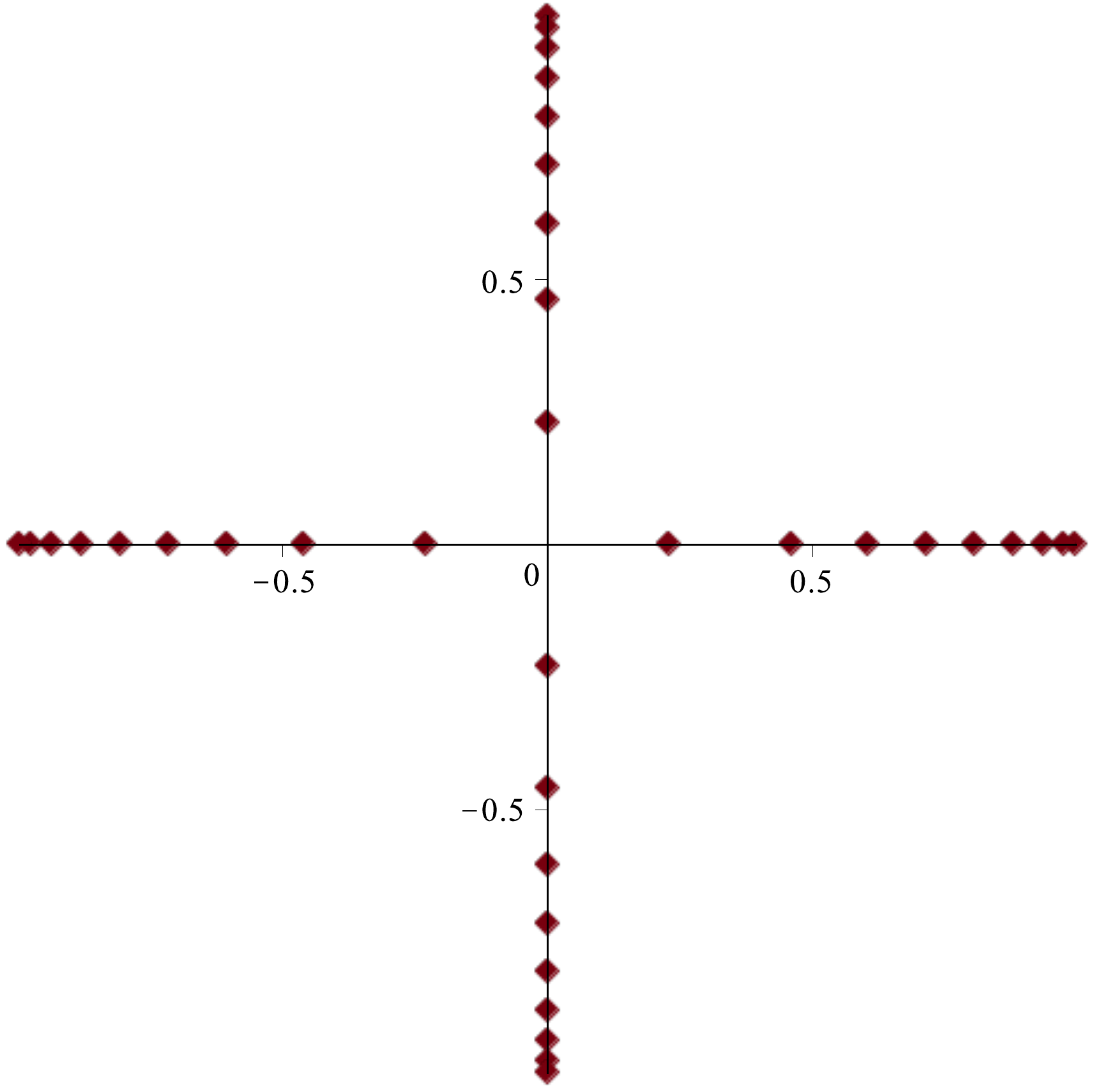}}~
\subfigure[]{\includegraphics[scale=.25]{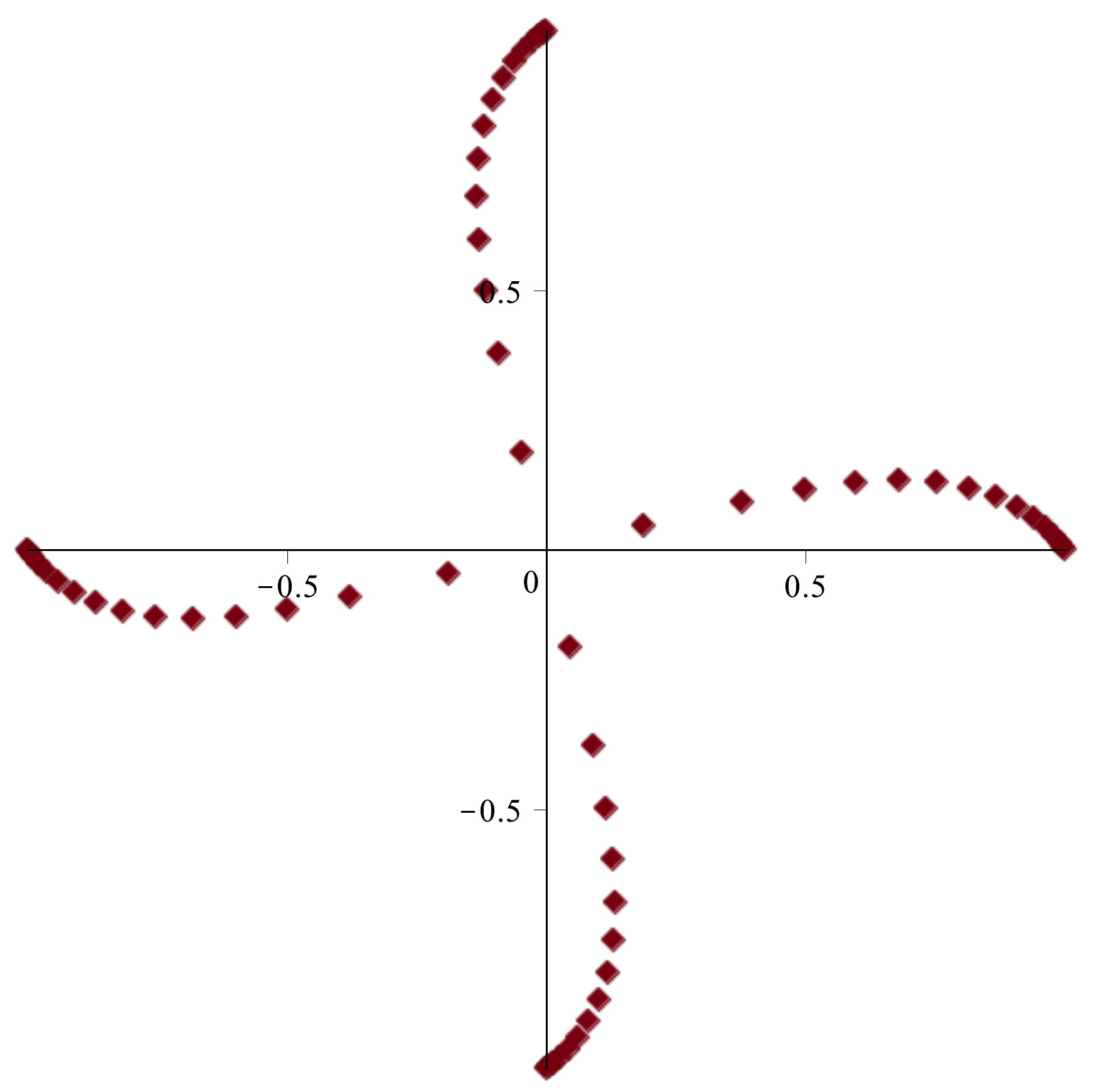}}~
\subfigure[]{\includegraphics[scale=.25]{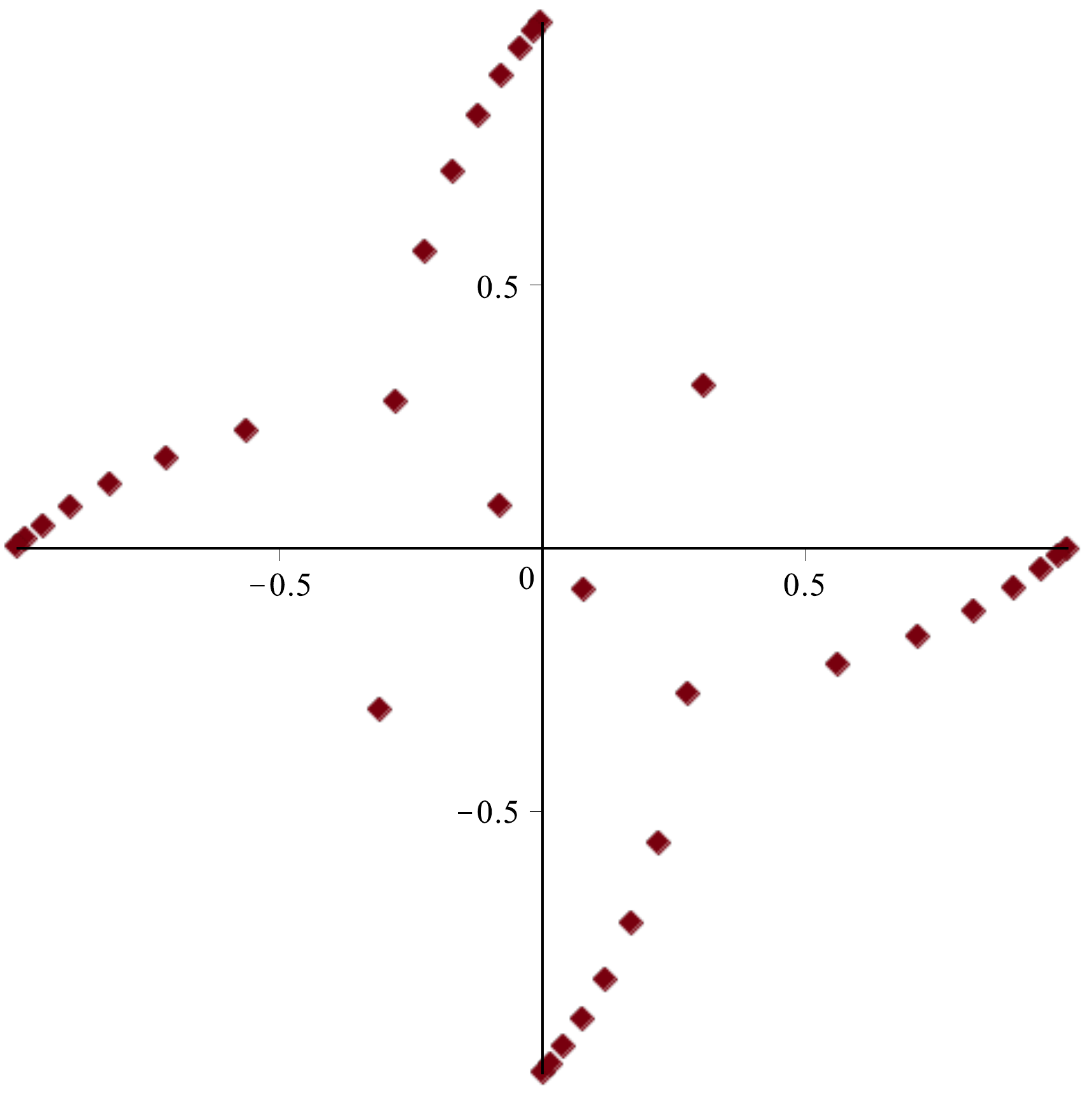}}
\caption{Poles of multipoint Pad\'e approximants to \( (z^4-1)^{-1/4} \) associated with interpolation schemes \( \mathcal V_* \), \( \mathcal V_*^\prime \), and \( \mathcal V_*^{\prime\prime} \).}
\label{fig:f4}
\end{figure}

Let now \( g(z) = (z^4-1)^{-1/4} \). Again, we have that \( \Delta_g=\Delta_\infty =  [-1,1]\cup[-\ic,\ic] \). Thus, if we use the interpolation scheme \( \mathcal V_* \), the poles of the corresponding multipoint  Pad\'e approximants must accumulate on \( [-1,1]\cup[-\ic,\ic] \), see Figure~\ref{fig:f4}(a) for the poles of \( [36/36;V_{72}]_g \). Likewise, the poles of \( [n/n;V_{2n}^\prime]_g \) and of \( [n/n;V_{2n}^\prime]_f \) accumulate on the same contour, see Figure~\ref{fig:f4}(b) for the poles of \( [60/60;V_{120}^\prime]_g \). However, the conjectural contour symmetric with respect to \( (\RS_*,V_*^{\prime\prime}) \) does not make \( g \) single-valued in its complement. That is, the surface \( \RS_* \) is no longer appropriate for the considered approximation problem. On Figure~\ref{fig:f4}(c) the poles of \( [34/34;V_{68}^{\prime\prime}]_g \) are depicted. It suggests that the appropriate surface should be given by \( \big\{(z,w):~w^2=(z^2-b^2)(z^4-1)\big\} \) for some unique \( b \). The genus of this surface is 2 and one can clearly see two poles with atypical behavior on Figure~\ref{fig:f4}(c).

\section{Symmetric Contours}
\label{sec:sym-proof}

The following observation will be important: if \( L \) is a smooth arc and \( g(z) \) is harmonic from each side of \( L \), \( g(s)=0 \) for \( s\in L \), and
\begin{equation}
\label{observation}
\frac{\partial g}{\partial\boldsymbol n_-} =\frac{\partial g}{\partial\boldsymbol n_+} \quad \text{on} \quad L^\circ,
\end{equation}
where \( L^\circ \) is \( L \) without the endpoints, then the harmonic continuation of \( g(z) \) across \( L^\circ \) is given by \( -g(z) \). Indeed, let \( H := \partial_z g \), where \( 2\partial_z := \partial_x - \ic\partial_y \). Then \( H(z) \) is a holomorphic function from each side of \( L \). Denote by \( \boldsymbol\tau(s) \) and \( \boldsymbol n_\pm(s) \) the unit tangent vector and the one-sided unit normal vectors to \( L^\circ \) at \( s\in L^\circ \). Further, put \( \tau(s) \) and \( n_\pm(s) \) to be the corresponding unimodular complex numbers, \( n_+(s)=\ic\tau(s) \). Then
\begin{multline*}
0 = \frac{\partial g}{\partial\boldsymbol\tau}(s) = \big\langle \nabla g(s), \boldsymbol\tau(s) \big\rangle = \mp2\im\big(n_\pm(s)H_\pm(s)\big) \quad \Rightarrow \\ \frac{\partial g}{\partial\boldsymbol n_\pm}(s) = \big\langle \nabla g(s), \boldsymbol n_\pm(s) \big\rangle = 2\re\big(n_\pm(s)H_\pm(s)\big) = n_\pm(s)H_\pm(s).
\end{multline*}
As \( n_+(s) = -n_-(s) \), \( -H(z) \) is the analytic continuation of \( H(z) \) across \( L^\circ \), which is equivalent to the original claim.

\subsection{Proof of Proposition~\ref{prop:connections}}

It follows from Theorem~\ref{thm:stahl} that \( \Delta_f \) is a branch cut for \( h_{D_f}(z) \) and the jump of \( h_{D_f}(z) \) across any subarc of \( \Delta_f \) is non-zero. According to the choice of the set \( E \), \( \RS \) is also the Riemann surface of \( h_{D_f}(z) \). Therefore, \( \Delta_f \) satisfies Definition~\ref{def:symmetry}(i). Realize \( \RS \) as in \eqref{realization} with \( \Delta \) and \( D \) replaced by \( \Delta_f \) and \( D_f \). Define \( g(\z) \) on \( \RS \) by lifting \( (-1)^ig_{D_f}(z) \) to \( D^{(i)} \) and extend it to \( \bd \) continuously by zero. Zeroing on \( \bd \), the S-property from Theorem~\ref{thm:stahl}(iii), and \eqref{observation} imply that \( g(\z) \) is harmonic across \( \bd \). Since only constants are harmonic on the entire surface \( \RS \), \( g(\z) = g(\z,\infty^{(0)}) \) and Definition~\ref{def:symmetry}(ii) follows.

To prove the second claim of the proposition, define \( \nu_n \) as in \eqref{asymp-measure}. Realize \( \RS \) as in \eqref{realization}. Define \( g_{\nu_n}(\z) \) on \( \RS \) by lifting \( (-1)^ig_D^{\nu_n}(z) \) to \( D^{(i)} \) and then extending it to \( \bd \) by continuity. It follows from the definition of the Green functions and Definition~\ref{def:g-fun} that
\[
g_{\nu_n}(\z) - g_n(\z), \quad g_n(\z) := \frac1{2n}\sum_{v\in V_{2n}}g\big(\z,v^{(0)} \big),
\]
is harmonic in each domain \( D^{(i)} \). Moreover, as \( g_{\nu_n} \equiv 0 \) on \( \bd \) and \( g_n=o(1) \) by Definition~\ref{def:symmetry}(ii), the above differences converge to zero uniformly on \( \RS \) by the maximum principle for harmonic functions. In another connection, let \( O \) be a neighborhood of \( \Delta \) such that \( O\cap V_{2n} = \varnothing \) for all \( n \). Then \( g_D^{\nu_n}(z) \) and \( g_D^\nu(z) \) are harmonic in \( O\setminus\Delta \). This and the weak\( ^* \) convergence imply that the functions \( g_D^{\nu_n}(z) \) converge to \( g_D^\nu(z) \) locally uniformly in \( O \). Define \( g(\z) \) by lifting \( (-1)^ig_D^\nu(z) \) to \( D^{(i)} \).  The previous two limits yield that \( g(\z) = g_n(\z) + o(1) \), where \( o(1) \) holds locally uniformly in \( \pi^{-1}(O) \). As the functions \( g_n(\z) \) are harmonic in \(  \pi^{-1}(O) \), so must be their uniform limit. This and claim \eqref{observation} finish the proof of the proposition.

\subsection{Functions \( g(\z,\ve) \)}

The main goal of this subsection is to show that
\begin{equation}
\label{flip}
g(\z,\ve) = g(\ve,\z).
\end{equation}
To this end, let us point out that the contour \( \Delta_v \) always exists (this is the first claim of Theorem~\ref{thm:scontours}). Indeed, let \( \bd_\ve \) be the zero level line of \( g(\cdot,\ve) \). Since the function \( g(\z,\ve) \) approaches \( + \infty \) as \( \z\to\ve \), approaches \( -\infty \) as \( \z\to\ve^* \), and is harmonic on \( \RS\setminus\{\ve,\ve^*\} \), \( \bd_\ve \) separates \( \RS \) into exactly two connected components. Symmetry \eqref{g-symmetry} then yields that \( \bd_\ve \) is involution-symmetric and the involution \( \cdot^* \) sends one component of \( \RS\setminus\bd_\ve \) into another. It only remains to notice that \( \Delta_v = \pi(\bd_\ve) \).

Fix \( \z \) and \( \ve \). Realize \( \RS \) as in \eqref{realization} with \( \Delta \) and \( D \) replaced by \( \Delta_v \) and \( D_v:=\overline\C\setminus\Delta_v \). Assume first that \( \z \) belongs to the same component of \( \RS\setminus\bd_\ve \) as \( \ve \), say \( D^{(0)} \). Then it easily follows from the properties of the Green function that
\begin{equation}
\label{flip1}
g\big(s^{(0)},v^{(0)}\big) = g_{D_v}(s,v), \quad s\in D_v.
\end{equation}
Denote by \( \partial D_v \) the boundary of \( D_v \) considered as a set of different accessible points from \( D_v \). Hence, every smooth point of \( \Delta_v \) appears twice in \( \partial D_v \) since it can be accessed from one side of the corresponding subarc of \( \Delta_v \) or the other. Let \( \big\{\widehat\delta_s\big\}_{s\in D_v} \) be the harmonic measure\footnote{If we denote by \( p \) the projection taking a point on \( \partial D_v \) (viewed as a set of accessible points) into the corresponding point on \( \Delta_v \), then the classically considered harmonic measure on \( \Delta_v \) is simply \( \big\{\widehat\delta_s\circ p^{-1}\big\}_{s\in D}\). } on \( \partial D_v \) (equivalently, \( \widehat\delta_s \) is the balayage of the Dirac delta distribution \( \delta(s) \) from \( D_v \) onto \( \partial D_v \), see \cite[Section II.4]{SaffTotik}). Then it holds that
\begin{eqnarray*}
g\big(s^{(0)},\z\big) & = & g_{D_v}(s,z) + \left( g\big(s^{(0)},\z\big) - g_D(s,z) \right) \\
& = & g_{D_v}(s,z) + \int_{\partial D_v} \left( g\big(t^{(0)},\z\big) - g_D(t,z) \right) \dd\widehat\delta_s(t) \\
\end{eqnarray*}
by the properties of harmonic measure. Moreover, since Green function is zero on the boundary of the domain, we get that
\begin{eqnarray*}
g\big(s^{(0)},\z\big) & = & g_{D_v}(s,z) + \int_{\partial D_v} g\big(t^{(0)},\z\big) \dd\widehat\delta_s(t) \\
& = & g_{D_v}(s,z) + \frac1{2\pi}\int_{\partial D_v} g\big(t^{(0)},\z\big) \frac{\partial g_{D_v}(\cdot,s)}{\partial\boldsymbol n}\big|_t\dd t,
\end{eqnarray*}
where \( \partial/\partial\boldsymbol n \) is the partial derivative with respect to the inner normal on \( \partial D_v \) and the second equality follows from \cite[Equation~(II.4.32) and Theorem~II.1.5]{SaffTotik}. Equivalently, using the one-sided normals \( \boldsymbol n_\pm \) on \( \Delta_v \), we can write
\begin{eqnarray*}
g\big(s^{(0)},\z\big) &=& g_{D_v}(s,z) + \frac1{2\pi}\int_{\Delta_v}\left( g_+\big(t^{(0)},\z\big) \frac{\partial g_{D_v}(\cdot,s)}{\partial\boldsymbol n_+}\big|_t + g_-\big(t^{(0)},\z\big) \frac{\partial g_{D_v}(\cdot,s)}{\partial\boldsymbol n_-}\big|_t \right)\dd t \\
& = & g_{D_v}(s,z) + \frac1{2\pi}\int_{\Delta_v}\left( \frac{\partial g_{D_v}(\cdot,s)}{\partial\boldsymbol n_+}\big|_t - \frac{\partial g_{D_v}(\cdot,s)}{\partial\boldsymbol n_-}\big|_t \right) g_+\big(t^{(0)},\z\big) \dd t,
\end{eqnarray*}
where the second equality follows from \eqref{g-symmetry}. When \( \s = \ve \), the last integral is equal to zero by claim \eqref{observation} and the very definition of \( \Delta_v \). Hence,
\begin{equation}
\label{flip2}
g\big(v^{(0)},z^{(0)}\big) = g_{D_v}(v,z)
\end{equation}
and the desired symmetry \eqref{flip} follows from \eqref{flip1}, \eqref{flip2}, and the well known symmetry of the Green function, see \cite[Theorem~II.4.9]{SaffTotik}. The cases when \( \z \) and \( \ve \) belong to different connected components of \( \RS\setminus\bd_\ve \) or \( \z\in\Delta_\ve \), can be shown similarly.

One consequence of \eqref{flip} is that \( g(\z,\ve) \) is a harmonic function of \( \ve\in\RS\setminus\{\z,\z^*\} \). Therefore, if \( \RS \) is realized as in \eqref{realization}, then for any closed set \( K\subset D \) there exists a constant \( C_K \) such that
\begin{equation}
\label{gE}
\big|g\big(\x,v_1^{(0)}\big)-g\big(\x,v_2^{(0)}\big)\big| \leq C_K|v_1-v_2|, \quad \x\in\bd, \quad v_1,v_2\in K.
\end{equation}

\subsection{Proof of Theorem~\ref{thm:scontours}}

The existence of \( \Delta_v \) was shown in the previous subsection. Thus, we only need to prove the second claim of the theorem. Set
\[
\Delta_\Xi := \Xi(\Delta) \qandq D_\Xi := \overline\C\setminus\Delta_\Xi,
\]
and realize \( \RS \) as in \eqref{realization} with \( \Delta \) replaced by \( \Delta_\Xi \). Denote by \( O \) the interior domain of \( \Xi(L_c) \). Define
\[
g\big(z^{(i)}\big) = (-1)^i g_{D_v}\big(\Xi^{-1}(z),v\big), \quad z\in O,
\]
where \( g_{D_v}(z,v) \) is the Green function with pole at \( v \) for \( D_v \). Then, as in the proof of Proposition~\ref{prop:connections},  \( g(\z) \) is harmonic in \( \pi^{-1}(O) \) and \( \pi^{-1}(L_c) \) necessarily consists of two level lines of \( g(\z) \). Assume that we can write
\begin{equation}
\label{representation}
g(\z) = \int_{\partial O}g(\z,s)\frac{\partial g}{\partial\boldsymbol n}(s) |\dd s|, \quad \z\in\pi^{-1}(O),
\end{equation}
where we adopt the convention \( f(z) := f\big(z^{(0)}\big) \) and \( f^*(z) := f\big(z^{(1)}\big) \), \( z\in D_\Xi \), for a function \( f(\z) \) on \( \RS \). Split \( \partial O \) into \( 2n \) disjoint (except for the endpoints) subarcs \( L_{2n,i} \) such that
\[
\int_{L_{2n,i}}\frac{\partial g}{\partial\boldsymbol n}(s) |\dd s| = \frac1{2n}\int_{\partial O}\frac{\partial g}{\partial\boldsymbol n}(s) |\dd s| =: \frac1{2nC}
\]
and pick \( v_{2n,i}\in L_{2n,i} \). Since \( g(\x) =0 \) for \( \x\in\bd \), we get from \eqref{representation} that
\begin{multline*}
g_n(\x) := \sum_{i=1}^{2n} g\big(\x,v_{2n,i}^{(0)}\big ) = 2nC \sum_{i=1}^{2n} \int_{L_{2n,i}} g\big(\x,v_{2n,i}^{(0)}\big)\frac{\partial g}{\partial\boldsymbol n}(v) |\dd v| = \\ = 2nC \sum_{i=1}^{2n} \int_{L_{2n,i}} \left( g\big(\x,v_{2n,i}^{(0)}\big) - g(\x,v) \right)\frac{\partial g}{\partial\boldsymbol n}(v) |\dd v|.
\end{multline*}
Hence, it holds that
\begin{eqnarray}
|g_n(\x)| & \leq & 2nC\left(\int_L\left|\frac{\partial g}{\partial\boldsymbol n}(v)\right||\dd v|\right) \max_i\max_{v_1,v_2\in L_{2n,i}}\left|g\big(\x,v_1^{(0)}\big) - g\big(\x,v_2^{(0)}\big)\right|, \nonumber \\
& \leq &  C^\prime n\max_i \diam(L_{2n,i}) \leq C^{\prime\prime}, \quad \x\in\bd, 
\label{E}
\end{eqnarray}
by \eqref{gE} for some constants \( C^\prime,C^{\prime\prime} \) that depend only on \( \Xi \). Define \(\nu_n \) as in \eqref{asymp-measure} with just selected sets \( V_{2n} = \{v_{2n,i}\}_{i=1}^{2n} \). The functions
\begin{equation}
\label{gandg}
g_n\big(z^{(1)}\big) + 2ng_{D_\Xi}^{\nu_n}(z)
\end{equation}
are harmonic in \( D_\Xi \) and have bounded traces on \( \Delta_\Xi \) according to \eqref{E}. By the maximum principle for harmonic functions they are uniformly bounded above and below in \( D_\Xi \). As the measures \( \nu_n \) are supported on \( \partial O \), which is compact, any sequence of them contains a weak\( ^* \) convergent subsequence by Helly's selection principle. Let \( \nu \) be the limit. Then
\[
g_{D_\Xi}^{\nu_n}(z) \to g_{D_\Xi}^\nu(z) \qasq n\to\infty, \quad z\in D_\Xi\setminus\partial O.
\]
As \( g_{D_\Xi}^\nu(z)> 0 \) in \( D_\Xi \), \( g\big(z^{(1)}\big)\to-\infty \) for every \( z\in D_\Xi\setminus\partial O \) by the conclusion after \eqref{gandg}. Therefore \( g\big(z^{(1)}\big)\to-\infty \) locally uniformly in \( D_\Xi \) by the maximum principle for subharmonic functions. This shows that the condition in Definition~\ref{def:symmetry}(ii) is fulfilled and thus finishes the proof of the theorem given representation \eqref{representation}.

To prove \eqref{representation}, let us recall the Green's formula stated in a form convenient for our purposes. Let \( U \) be an open set with piecewise smooth boundary and let \( a,b\) be two harmonic functions in \( U \) with piecewise smooth traces on \( \partial U \). Then
\begin{equation}
\label{Green}
\int_{\partial U} a(s)\frac{\partial b}{\partial\boldsymbol n}(s)|\dd s| = \int_{\partial U} b(s)\frac{\partial a}{\partial\boldsymbol n}(s)|\dd s|,
\end{equation}
where \( \partial / \partial\boldsymbol n \) is the partial derivative with respect to the inner normal on \( \partial U \) and \( |\dd s| \) is the arclength differential.

Given distinct \( \x,\y \in \RS\setminus\{\e_0\} \), denote by \( g(\cdot,\x,\y) \) a function that is harmonic in \( \RS\setminus\{ \x,\y\} \), normalized so that \( g(\e_0,\x,\y ) = 0 \), and such that
\[
g(\z,\x,\y) + \left\{\begin{array}{rl}
\log|z-x|, & |x|<\infty, \smallskip \\
-\log|z|, & x=\infty,
\end{array}
\right.
\qandq g(\z,\x,\y) - \left\{\begin{array}{rl}
\log|z-y|, & |y|<\infty, \smallskip \\
-\log|z|, & y=\infty,
\end{array}
\right.
\]
are harmonic around \( \x \) and \( \y \), respectively. Existence of such functions follows from the  same principles as the existence of \( g(\z,\ve) \) in Definition~\ref{def:g-fun}. Fix \( \z\in \pi^{-1}(O)\setminus\bd \) and denote by \( U \) a disk centered  at \( z \) of radius \( \delta>0 \) small enough so that \( \overline U\subset O\setminus\Delta_\Xi \). Then, assuming \( \z\in D^{(0)} \), it holds that
\begin{eqnarray*}
I & := & \int_{\partial (O\setminus \overline U)} g(s)\frac{\partial}{\partial\boldsymbol n} g\big(s,\z,\infty^{(1)}\big) |\dd s| + \int_{\partial O} g^*(s) \frac{\partial}{\partial\boldsymbol n} g^*\big(s,\z,\infty^{(1)}\big) |\dd s| \\
& = & c\int_{\partial O} \frac{\partial}{\partial\boldsymbol n} g\big(s,\z,\infty^{(1)}\big) |\dd s| - c\int_{\partial O} \frac{\partial}{\partial\boldsymbol n} g^*\big(s,\z,\infty^{(1)}\big) |\dd s| - \int_{\partial U} g(s)\frac{\partial}{\partial\boldsymbol n} g\big(s,\z,\infty^{(1)}\big) |\dd s|.
\end{eqnarray*}
Observe that \( g\big(\cdot,\z,\infty^{(1)}\big) \) is harmonic outside of \( \overline O \) and therefore
\[
\int_{\partial O} \frac{\partial}{\partial\boldsymbol n} g\big(s,\z,\infty^{(1)}\big) |\dd s| = 0
\]
by \eqref{Green}. Analogously, \eqref{Green} and the definition of \( g(\cdot,\x,\y) \) yield that
\[
\int_{\partial O} \frac{\partial}{\partial\boldsymbol n} g^*\big(s,\z,\infty^{(1)}\big) \dd s = -\int_{\partial O} \frac{\partial}{\partial\boldsymbol n} \log|s| |\dd s|  = \int_{|s|=r} \frac{\partial}{\partial r} \log r |\dd s| = 2\pi.
\]
for any \( r \) large. Furthermore, we have
\[
\int_{\partial U} g(s)\frac{\partial}{\partial\boldsymbol n} g\big(s,\z,\infty^{(1)}\big) |\dd s| =: I_* - \int_{\partial U} g(s)\frac{\partial}{\partial\boldsymbol n} \log|s-z| |\dd s| = I_* - \frac1\delta\int_{\partial U} g(s)|\dd s|.
\]
Thus, we get from the mean-value property of harmonic functions that
\begin{equation}
\label{I1}
I = -2\pi c + g(\z) - I_*.
\end{equation}
In another connection, since \( g(\x) =0 \) for \( \x\in\bd \), we deduce from \eqref{Green} that
\begin{eqnarray*}
I & := & \int_{\partial (O\setminus (\overline U\cup\Delta_\Xi))} g(s)\frac{\partial}{\partial\boldsymbol n} g\big(s,\z,\infty^{(1)}\big) |\dd s| + \int_{\partial (O\setminus\Delta_\Xi)} g^*(s) \frac{\partial}{\partial\boldsymbol n} g^*\big(s,\z,\infty^{(1)}\big) |\dd s| \\
& = & \int_{\partial (O\setminus (\overline U\cup\Delta_\Xi))} g\big(s,\z,\infty^{(1)}\big) \frac{\partial g}{\partial\boldsymbol n}(s) |\dd s| + \int_{\partial (O\setminus\Delta_\Xi)} g^*\big(s,\z,\infty^{(1)}\big) \frac{\partial g^*}{\partial\boldsymbol n}(s) |\dd s| \\
& = & \int_{\partial O} \left(g\big(s,\z,\infty^{(1)}\big) - g^*\big(s,\z,\infty^{(1)}\big) \right) \frac{\partial g}{\partial\boldsymbol n}(s) |\dd s| - \int_{\partial U} g\big(s,\z,\infty^{(1)}\big) \frac{\partial g}{\partial\boldsymbol n}(s) |\dd s|,
\end{eqnarray*}
where we also used the fact that \( g\big(s,\z,\infty^{(1)}\big) = g^*\big(s,\z,\infty^{(1)}\big) \) for \( s\in \Delta_\Xi \) while \( g(z) = - g^*(z) \). Clearly, it holds that
\[
g\big(s,\z,\infty^{(1)}\big) - g^*\big(s,\z,\infty^{(1)}\big) = g(s,\z ) + g\big(s,\infty^{(0)}\big) = g(\z,s) + g\big(s,\infty^{(0)}\big)
\]
by \eqref{flip}. Using \eqref{g-symmetry} and the symmetry of \( g(\z) \), we get that
\begin{eqnarray*}
\int_{\partial O} g\big(s,\infty^{(0)}\big) \frac{\partial g}{\partial\boldsymbol n}(s) |\dd s| & = & \frac12\int_{\partial O} \left(g\big(s,\infty^{(0)}\big) \frac{\partial g}{\partial\boldsymbol n}(s) + g^*\big(s,\infty^{(0)}\big) \frac{\partial g^*}{\partial\boldsymbol n}(s)\right) |\dd s| \\
& = & \frac12\int_{\partial (O\setminus\Delta_\Xi)} \left(g\big(s,\infty^{(0)}\big) \frac{\partial g}{\partial\boldsymbol n}(s) + g^*\big(s,\infty^{(0)}\big) \frac{\partial g^*}{\partial\boldsymbol n}(s)\right) |\dd s|.
\end{eqnarray*}
Then, it follows from \eqref{Green} that
\begin{eqnarray*}
\int_{\partial O} g\big(s,\infty^{(0)}\big) \frac{\partial g}{\partial\boldsymbol n}(s) |\dd s|  & = & \frac12\int_{\partial (O\setminus\Delta_\Xi)} \left( g(s) \frac{\partial }{\partial\boldsymbol n}g\big(s,\infty^{(0)}\big)  + g^*(s)\frac{\partial }{\partial\boldsymbol n}g^*\big(s,\infty^{(0)}\big)\right) |\dd s| \\
& = & \frac c2 \int_{\partial O}\left( \frac{\partial }{\partial\boldsymbol n}g\big(s,\infty^{(0)}\big)  - \frac{\partial }{\partial\boldsymbol n}g^*\big(s,\infty^{(0)}\big)\right) |\dd s| \\
& = & c \int_{\partial O}\frac{\partial }{\partial\boldsymbol n}\log|s| |\dd s| = -2\pi c.
\end{eqnarray*}
Moreover, it holds that
\begin{multline*}
\int_{\partial U} g\big(s,\z,\infty^{(1)}\big) \frac{\partial g}{\partial\boldsymbol n}(s) |\dd s| = -\log\delta \int_{\partial U}\frac{\partial g}{\partial\boldsymbol n}(s) |\dd s| + \\ \int_{\partial U} \left(g\big(s,\z,\infty^{(1)}\big) + \log|s-z|\right)\frac{\partial g}{\partial\boldsymbol n}(s) |\dd s| = I_*,
\end{multline*}
again by \eqref{Green}. Altogether, we have showed that
\begin{equation}
\label{I2}
I = \int_{\partial O}g(\z,s)\frac{\partial g}{\partial\boldsymbol n}(s) |\dd s| - 2\pi c - I_*.
\end{equation}
Hence, by combining \eqref{I1} and \eqref{I2}, we get \eqref{representation} for \( \z\in D^{(0)} \). Clearly, the proof for \( \z\in D^{(1)} \) is absolutely analogous, which then allows us to extend  \eqref{representation} to \( \bd \) by continuity.

\section{Nuttall-Szeg\H{o} Functions} 
\label{sec:NS}

In what follows, we set \( \RS_{\ualpha,\ubeta} := \RS\setminus\cup_{k=1}^g(\ualpha_k\cup\ubeta_k) \) and \( \RS_{\ualpha} := \RS\setminus\cup_{k=1}^g\ualpha_k \), where \( \{\ualpha_k,\ubeta_k\} \) is the chosen homology basis. When \( g=0 \), we have that \( \RS_{\ualpha,\ubeta} = \RS_{\ualpha} = \RS \).

\subsection{Riemann Theta Function}

Let \( \am \) be Abel's map defined in \eqref{abel-map}. Specializing divisors to a single point \( \z \), \( \am(\z) \) becomes a vector of holomorphic functions in \( \RS_{\ualpha,\ubeta} \) with continuous traces on the cycles of the homology basis that satisfy
\begin{equation}
\label{abel-jump}
\am_+(\s) - \am_-(\s) = \left\{
\begin{array}{rl}
-\boldsymbol B \vec e_k, & \s\in\ualpha_k, \medskip \\
\vec e_k, & \s\in\ubeta_k,
\end{array}
\right. \quad k\in\{1,\ldots,g\},
\end{equation}
by \eqref{B} and the normalization of \( \vec\Omega \). It readily follows from \eqref{abel-jump} that each \( (\am)_k \) is, in fact, holomorphic in \( \RS_{\ualpha}\setminus\ubeta_k \).

The theta function associated with \( \boldsymbol B \) is an entire transcendental function of \( g \) complex variables defined by
\[
\theta\left(\vec u\right) := \sum_{\vec n\in\Z^g}\exp\bigg\{\pi\mathrm{i}\vec n^\mathsf{T}\boldsymbol B\vec n + 2\pi\mathrm{i}\vec n^\mathsf{T}\vec u\bigg\}, \quad \vec u\in\C^g.
\]
As shown by Riemann, the symmetry of \( \boldsymbol B \) and positive definiteness of its imaginary part ensures the convergence of the series for any \( \vec u \). It can be directly checked that \( \theta(-\vec u)=\theta(\vec u) \) and it enjoys the following periodicity property:
\begin{equation}
\label{theta-periods}
\theta\left(\vec u + \vec j + \boldsymbol B\vec m\right) = \exp\bigg\{-\pi \mathrm{i}\vec m^\mathsf{T}\boldsymbol B \vec m - 2\pi \mathrm{i}\vec m^\mathsf{T}\vec u\bigg\}\theta\big(\vec u\big), \quad \vec j,\vec m\in\Z^g.
\end{equation}
It is also known that \( \theta\left(\vec u\right)=0 \) if and only if \( \vec u\equiv \am\left(\mathcal{D}_{\vec u}\right) + \vec K \) \( \big(\mdp \vec\Omega\big) \) for some effective divisor \( \mathcal{D}_{\vec u} \) of degree \( g-1 \) depending on \( \vec u \), where \( \vec K \) is a fixed  vector known as the vector of Riemann constants (it can be explicitly defined via \( \vec\Omega \)).

Assume that \( \mathcal D_n \) is the unique solution\footnote{Recall that otherwise it would contain an involution-symmetric pair of points. Then, as \( \am(\s) + \am(\s^*) \equiv 0 \), the expression \(  \am(\mathcal{D}_n) + \vec K -\am(\z) \) would belong to the zero set of \( \theta \) for any \( \z \).} of \eqref{main-jip}. Set
\begin{equation}
\label{thetan}
\Theta_n(\z) := \theta\left(\am(\z) - \am(\mathcal{D}_n) - \vec K\right) / \theta\left(\am(\z) - \am\big(g\infty^{(1)}\big) - \vec K\right).
\end{equation}
The function \( \Theta_n \) is a multiplicatively multi-valued meromorphic function on \( \RS \) with zeros at the points of the divisor \( \mathcal D_n \) of respective multiplicities, a pole of order \( g \) at \( \infty^{(1)} \), and otherwise non-vanishing and finite (there will be a reduction of the order of the pole at \( \infty^{(1)} \) when the divisor \( \mathcal D_n \) contains this point). In fact, it is meromorphic and single-valued in \( \RS_{\ualpha} \) and
\begin{eqnarray}
\Theta_{n+}  &=& \Theta_{n-} \exp\left\{2\pi \mathrm{i}\big((\am)_k\big(g\infty^{(1)}\big) - (\am)_k(\mathcal D_n)\big)\right\} \nonumber \\
\label{jump2}
&=& \Theta_{n-} \exp\left\{-2\pi \mathrm{i}\left(\vec c_\rho + \vec\omega_n + \boldsymbol B\vec\tau_n +\boldsymbol B\vec m_n\right)_k\right\} \quad \text{on} \quad \ualpha_k
\end{eqnarray}
by \eqref{theta-periods} and \eqref{abel-jump}, where \( \vec m_n,\vec j_n\in\Z^g \) are such that
\begin{equation}
\label{jnmn}
\am(\mathcal D_n)-\am\big(g\infty^{(1)}\big) = \vec c_\rho + \vec\omega_n + \boldsymbol B\vec\tau_n + \vec j_n +\boldsymbol B\vec m_n.
\end{equation}

\subsection{Szeg\H{o}-type Functions on \( \RS \)}

Let \( G_\ve \) be as defined before \eqref{constants}. Consider the differential
\begin{equation}
\label{Gn}
G_n(\z) := \frac12\sum_{|v_{2n,i}|<\infty}\left(\frac{\dd z}{z-v_{2n,i}} + G_{v_{2n,i}^{(0)}}(\z)\right) + \frac{2n-\deg(v_{2n})}2 G_{\infty^{(0)}}(\z).
\end{equation}
It is holomorphic on \( \RS \) except for a pole at every \( v_{2n,i}^{(1)} \), \( |v_{2n,i}|<\infty \), with residue equal to the multiplicity of \( v_{2n,i} \) in \( V_{2n} \), a pole at \( \infty^{(1)} \)  with residue \( n-\deg(v_{2n}) \),  and a pole at \( \infty^{(0)} \)  with residue \( -n \). Furthermore, since the cycles of the homology basis are involution-symmetric, it holds that
\begin{equation}
\label{periods}
(\vec \omega_n)_k = -\frac1{2\pi\mathrm i}\oint_{\ubeta_k} G_n \qandq (\vec \tau_n)_k = \frac1{2\pi\mathrm i}\oint_{\ualpha_k} G_n,
\end{equation}
\( k\in\{1,\ldots,g\} \), where the vectors \( \vec\omega_n,\vec\tau_n \) were defined in \eqref{constants}. Put
\begin{equation}
\label{Sn}
S_n(\z) := \exp\left\{\int_{\e_0}^\z G_n\right\} \left\{
\begin{array}{rl}
1, & \z\in D^{(0)}, \medskip \\
v_{2n}^{-1}(z), & \z\in D^{(1)}.
\end{array}
\right.
\end{equation}
Then \( S_n \) is a meromorphic in \( \RS_{\ualpha,\ubeta}\setminus\bd \) function with a pole of order \( n \) at \( \infty^{(0)} \), a zero of order \( n \) at \( \infty^{(1)} \), otherwise non-vanishing and finite, and such that
\begin{equation}
\label{jumpa}
S_{n+}(\x) = S_{n-}(\x)\left\{
\begin{array}{ll}
v_{2n}(x), & \x\in\bd, \medskip \\
\exp\left\{2\pi\mathrm{i}\big(\vec\omega_n\big)_k\right\}, & \x\in\ualpha_k, \medskip \\
\exp\left\{2\pi\mathrm{i}\big(\vec\tau_n\big)_k\right\}, & \x\in\ubeta_k,
\end{array}
\right. \quad k\in\{1,\ldots,g\}.
\end{equation}

Let \( \ugamma \) be an involution-symmetric, piecewise-smooth oriented chain on \( \RS \) that has only finitely many points in common with the \( \ualpha \)-cycles. Further, let $\lambda$ be a H\"older continuous function on \( \ugamma \). Denote by \( \Omega_{\z,\z^*} \) the normalized abelian differential of the third kind\footnote{It is a meromorphic differential with two simple poles at \( \z \) and \( \z^* \) with respective residues \( 1 \) and \( -1 \) normalized to have zero periods on the \( \ualpha \)-cycles.}. Set
\begin{equation}
\label{Lambda}
\Lambda(\z) := \frac1{4\pi\mathrm{i}}\oint_{\ugamma}\lambda\Omega_{\z,\z^*}, \quad \z\not\in\ugamma.
\end{equation}
It is known \cite[Eq.~(2.7)--(2.9)]{Zver71} that \( \Lambda \) is a holomorphic function in \( \RS_{\ualpha}\setminus\ugamma \), \( \Lambda(\z)+\Lambda(\z^*)\equiv0 \) there, the traces \( \Lambda_\pm \) are H\"older continuous and satisfy
\[
\Lambda_+(\z) - \Lambda_-(\z) = \frac12\left\{
\begin{array}{rl}
\displaystyle \lambda(\z)+\lambda(\z^*), & \z\in\ugamma, \medskip \\
\displaystyle -2\oint_{\ugamma} \lambda\Omega_k, & \z\in\ualpha_k\setminus\ugamma.
\end{array}
\right.
\]
That is, the differential \( \Omega_{\z,\z^*} \) is a discontinuous Cauchy kernel on \( \RS \) (it is discontinuous because \( \Lambda \) has additional jumps across the \( \ualpha \)-cycles).

Let \( \rho \) be a non-vanishing H\"older continuous function on \( \Delta \). Select a smooth branch of \( \log\rho \) and lift it \( \bd \), \( \lambda_\rho:=-\log\rho\circ\pi \). Define
\begin{equation}
\label{Sp}
S_\rho(\z) := \exp\big\{\Lambda_\rho(\z)\big\}, \qquad \vec c_\rho := -\frac1{2\pi\mathrm{i}}\oint_{\bd}\lambda_\rho\vec\Omega.
\end{equation}
Then $S_\rho$ is a holomorphic and non-vanishing function in $\RS_{\ualpha}\setminus\bd$ with continuous traces that satisfy
\begin{equation}
\label{jumpSp}
S_{\rho+}(\x) = S_{\rho-}(\x) \left\{
\begin{array}{ll}
\displaystyle \exp\big\{2\pi \mathrm{i}\big(\vec c_\rho\big)_k\big\}, & \x\in\ualpha_k, \medskip \\
1/\rho(x), & \x\in\bd.
\end{array}
\right.
\end{equation}

Next, let \( \vec m_n \) be defined by \eqref{jnmn}. Set \( \lambda_{\vec \tau_n} \) and \( \lambda_{\vec m_n} \) to be the functions on \( \ugamma=\cup\ubeta_k \) such that \( \lambda_{\vec \tau_n}\equiv-2\pi\mathrm{i}(\vec\tau_n)_k \) and \( \lambda_{\vec m_n}\equiv-2\pi\mathrm{i}(\vec m_n)_k \) on \( \ubeta_k \). Put
\begin{equation}
\label{stau}
S_{\vec \tau_n}(\z) := \exp\big\{\Lambda_{\vec\tau_n}(\z)\big\} \qandq  S_{\vec m_n}(\z) := \exp\big\{\Lambda_{\vec m_n}(\z)\big\}
\end{equation}
for \( \z\in\RS_{\ualpha,\ubeta} \). Both functions are holomorphic in \( \RS_{\ualpha,\ubeta} \) with continuous traces on the cycles of the homology basis that satisfy
\begin{equation}
\label{staujumps}
S_{\vec \tau_n+}(\x) = S_{\vec \tau_n-}(\x) \left\{
\begin{array}{ll}
\displaystyle \exp\big\{2\pi \mathrm{i}\big(\boldsymbol B\vec\tau_n\big)_k\big\}, & \x\in\ualpha_k, \medskip \\
\displaystyle \exp\big\{-2\pi \mathrm{i}\big(\vec\tau_n\big)_k\big\}, & \x\in\ubeta_k,
\end{array}
\right.
\end{equation}
where the first equality follows straight from \eqref{B}, and
\begin{equation}
\label{jump3}
S_{\vec m_n+}(\x) = S_{\vec m_n-}(\x) \exp\big\{2\pi \mathrm{i}\big(\boldsymbol B\vec m_n\big )_k\big\}, \quad \x\in\ualpha_k,
\end{equation}
where there are no jump across the \( \ubeta \)-cycles as each \( (\vec m_n)_k \) is an integer.

\subsection{Functions \( \Psi_n \)}

Given the functions \eqref{thetan}, \eqref{Sn}, \eqref{Sp}, \eqref{stau}, and an arbitrary constant \( C_n \), the product
\begin{equation}
\label{Psi}
\Psi_n:=C_n S_nS_\rho S_{\vec\tau_n}S_{\vec m_n}\Theta_n
\end{equation}
is a sectionally meromorphic function in \( \RS\setminus\bd \) with the divisor \eqref{PsiZeros} whose traces satisfy \eqref{PsiBoundary} by \eqref{jump2}, \eqref{jumpa}, \eqref{jumpSp}, \eqref{staujumps}, and \eqref{jump3}.

To show uniqueness, assume that there exists \( \Psi \) satisfying \eqref{PsiBoundary} and whose divisor is given by \( (n-g)\infty^{(1)}+\mathcal{D}-n\infty^{(0)} \) for some effective divisor \( \mathcal D \). Then \( \Psi/\Psi_n \) is a rational function on \( \RS \) with the divisor \(\mathcal D - \mathcal D _n \). Therefore, the degree of \( \mathcal D \) is \( g \), in which case \(\Psi/\Psi_n \) is the lift of a rational function on \( \overline \C \). As \( \mathcal D_n \) solves \eqref{main-jip} uniquely, it has no involution-symmetric pairs. Hence, \( \mathcal D = \mathcal D _n \) and therefore \( \Psi/\Psi_n \) is a constant.

This finishes the proof of Proposition~\ref{prop:nuttall-szego}. Let us now prove the first estimate in \eqref{maxPsi}. Put \( C_n= |v_{2n}(e_0)|^{1/2} \) in \eqref{Psi}. As mentioned after Definition~\ref{def:g-fun}, its holds that
\[
g(\z,\ve) = \re\left(\int_{\e_0}^\z G_\ve\right).
\]
Thus, it follows from \eqref{Gn} and \eqref{Sn} that
\begin{equation}
\label{explicit}
\big| C_nS_n\big(z^{(1)}\big) \big| = \exp\left\{ \frac12\sum_{i=0}^{2n} g\big(z^{(1)},v_{2n,i}^{(0)}\big)\right\} / \big|v_{2n}(z) \big|^{1/2}, \quad z\in D.
\end{equation}
Further,  as \( \boldsymbol B \) has a positive definite imaginary part, any vector \( \vec u\in\C^g \) can be uniquely written as \( \vec x+\boldsymbol B\vec y \) for some \( \vec x,\vec y\in\R^g \). Since the image of the closure of \( \RS_{\ualpha,\ubeta} \) under Abel's map is bounded in \( \C^g \), so are the vectors \( \vec\omega_n+\vec j_n \) and \( \vec\tau_n+\vec m_n \) by \eqref{jnmn}. Therefore, 
\begin{equation}
\label{explicit2}
\left|S_\rho S_{\vec\tau_n}S_{\vec m_n}\right| \leq C
\end{equation}
uniformly with \( n \) in \( \RS_{\ualpha,\ubeta} \) for some absolute constant \( C \). Denote by \( \mathfrak D\) the closure of \( \{\mathcal D_n\}_{n\in\N_*} \) in \( \RS^g/\Sigma_g \)-topology. Associate to each \( \mathcal D\in\mathfrak D \) a function \( \Theta_{\mathcal D} \) defined as in \eqref{thetan} with \( \mathcal D_n \) replaced by \( \mathcal D \). The functions \( \Theta_{\mathcal D}/w \) are holomorphic in \( D^{(1)} \) and continuously depend on \( \mathcal D \). Therefore, they form a normal family \( D^{(1)} \), i.e., for any bounded set \( K\subset D \) there exists a constant \( C_K(\mathfrak D) \) such that
\begin{equation}
\label{explicit3}
\big| \Theta_n\big(z^{(1)}\big)/w(z) \big| \leq C_K(\mathfrak D), \quad n\in \N_*, \quad z\in K.
\end{equation}
Estimates \eqref{explicit}-\eqref{explicit3} immediately yield the first estimate in \eqref{maxPsi}. Observe also that the argument leading to \eqref{explicit3}, in fact, shows that the sequence \( \big\{ |\Theta_n| \big\} \) is uniformly bounded above on any closed subset of \( \RS\setminus\big\{ \infty^{(1)} \big\} \). Therefore, it holds that
\begin{equation}
\label{second-est}
\big| \Psi_n\big(z^{(1)}\big) \big| \leq \widetilde C_O(\mathfrak D) \exp\left\{ \frac12\sum_{i=0}^{2n} g\big(z^{(1)},v_{2n,i}^{(0)}\big) \right\}, \quad n\in \N_*, \quad z\in \overline O,
\end{equation}
for any open bounded set \( O\subset D \).

\subsection{Functions \( \Psi_n\Upsilon_n \)}

We start with the proof of Proposition~\ref{prop:secondary}. We are looking for a rational function \( \Upsilon_n \) with the divisor of the form \( \widetilde{\mathcal D}_n + \infty^{(0)} - \mathcal D_n - \infty^{(1)} \) for some effective divisor \( \widetilde{\mathcal D}_n \) of degree \( g \). By Abel's theorem, it must hold that
\begin{equation}
\label{secondary-jip}
\am(\widetilde{\mathcal D}_n)\, \equiv \, \am\big(\mathcal D_n + \infty^{(1)} - \infty^{(0)})  \quad \big(\mdp \vec\Omega\:\big).
\end{equation}
The above Jacobi inversion problem is always solvable and \( \Upsilon_n \) is unique up to a multiplicative factor if and only if the solution of \eqref{secondary-jip} is unique. If it were not, it would contain some and therefore any involution-symmetric pair. In particular, there would exist a solution containing \( \infty^{(1)} \). As \( \mathcal D_n \) has no involution-symmetric pairs, Abel's theorem and \eqref{secondary-jip} would yield that \( \mathcal D_n \) contains \( \infty^{(0)} \), which is impossible by the conditions of the proposition. This argument also shows that \( \Upsilon_n \) can have only a simple pole at \( \infty^{(1)} \).

It only remains to prove the second estimate in \eqref{maxPsi}. We shall show that \( \Psi_n\Upsilon_n \) admits a decomposition similar to \eqref{Psi}. To this end, denote by \( \widetilde{\mathfrak D} \) the closure of \( \{\widetilde{\mathcal D}_n\}_{n\in\N_*} \) in \( \RS^g/\Sigma_g \)-topology. Then \( \widetilde{\mathfrak D} \) has no divisors containing involution-symmetric pairs nor \( \infty^{(1)} \). The proof of this fact is exactly the same as in Proposition~\ref{prop:secondary}, where we use compactness of \( \RS^g/\Sigma_g \) and continuity of Abel's map to go from sequences to their limit points. Further, put
\[
\Phi(\z) := \exp\left\{\int_{\e_0}^\z G_{\infty^{(0)}}\right\}.
\]
Define the real vectors \( \vec\omega,\vec\tau \) by \eqref{periods} with \( G_n \) replaced by \( G_{\infty^{(0)}} \). Then, as in the case of \eqref{jumpa}, it holds that
\[
\Phi_+(\x) = \Phi_-(\x)\left\{
\begin{array}{ll}
\exp\left\{2\pi\mathrm{i}(\vec\omega)_k\right\}, & \x\in\ualpha_k, \medskip \\
\exp\left\{2\pi\mathrm{i}(\vec\tau\:)_k\right\}, & \x\in\ubeta_k.
\end{array}
\right.
\]
Notice that the differentials \( G_{\infty^{(0)}} \) and \( \Omega_{\infty^{(1)},\infty^{(0)}} \) have the same poles with the same residues. Thus, they differ by a holomorphic differential. From the normalization on the \( \ualpha \)-cycles we see that
\[
G_{\infty^{(0)}} = \Omega_{\infty^{(1)},\infty^{(0)}} + 2\pi\mathrm i\sum_{k=1}^g (\vec \tau\:)_k\Omega_k.
\]
Then it follows from Riemann's relations and \eqref{B} that
\[
(\am)_k\big(\infty^{(1)}-\infty^{(0)}\big) = \frac1{2\pi\mathrm i}\oint_{\ubeta_k}\Omega_{\infty^{(1)},\infty^{(0)}} = -(\vec\omega)_k - (\boldsymbol B\vec\tau)_k.
\]
Hence, we deduce from \eqref{main-jip} and \eqref{secondary-jip} that
\[
\am(\widetilde{\mathcal D}_n)\, \equiv \, \am\big(g\infty^{(1)}\big) + \vec c_\rho + (\vec\omega_n-\vec\omega) + \boldsymbol B(\vec\tau_n-\vec\tau)  \quad \big(\mdp \vec\Omega\:\big).
\]
Let \( \vec l_n,\vec k_n \) be defined by
\[
\am(\widetilde{\mathcal D}_n)-\am\big(g\infty^{(1)}\big) = \vec c_\rho + \vec\omega_n - \vec\omega + \boldsymbol B(\vec\tau_n -\vec\tau) + \vec l_n +\boldsymbol B\vec k_n.
\]
As before, it holds that \( \vec\tau_n -\vec\tau + \vec k_n \) is a bounded sequence of vectors and therefore so is \( \vec m_n - \vec k_n \). Then it can be verified as in the proof of Proposition~\ref{prop:nuttall-szego} that
\begin{equation}
\label{PsiUpsilon}
\Psi_n\Upsilon_n = C_nS_n S_\rho S_{\vec\tau_n-\vec\tau}S_{\vec k_n}\widetilde \Theta_n\Phi,
\end{equation}
where \( \widetilde \Theta_n \) is defined as in \eqref{thetan} with \( \mathcal D_n \) replaced by \( \widetilde{\mathcal D}_n \).The proof of the second estimate in \eqref{maxPsi} is now exactly the same as the proof of the first. Moreover, as in the case of \( \Psi_n \), it also holds that
\begin{equation}
\label{second-est2}
\big| (\Psi_n\Upsilon_n)\big(z^{(1)}\big) \big| \leq \widetilde C_O(\mathfrak D)\exp\left\{ \frac12\sum_{i=0}^{2n} g\big(z^{(1)},v_{2n,i}^{(0)}\big) \right\}, \quad n\in \N_*, \quad z\in \overline O,
\end{equation}
for any open bounded set \( O\subset D \).  

\subsection{Normalizing Constants}

Define
\begin{equation}
\label{gammas}
1/\gamma_n := \lim_{z\to\infty}\Psi_n\big(z^{(0)}\big)z^{-n} \qandq 1/\gamma_n^* := \lim_{z\to\infty}(\Psi_n\Upsilon_n)\big(z^{(1)}\big)z^{n-g-1}.
\end{equation}
The previous considerations imply that both constants are non-zero and finite when \( n\in\N_* \). Furthermore, it holds that
\begin{equation}
\label{gammas-asymp}
C^{-1}(\N_*) \leq |\gamma_n\gamma_n^*| \leq C(\N_*)
\end{equation}
for some constant \( C(\N_*) \). Indeed, we get from \eqref{g-symmetry} and \eqref{explicit} that
\[
\big|C_n^2S_n(\z)S_n(\z^*)\big| \equiv1, \quad z\in\RS.
\]
Recall also that the function \( \Lambda \) from \eqref{Lambda} was such that \( \Lambda(\z)+\Lambda(\z^*)\equiv0 \). Therefore,
\[
(S_\rho S_{\vec\tau_n})(\z)(S_\rho S_{\vec\tau_n})(\z^*) \equiv 1, \quad \z\in\RS.
\]
Similarly, it is easy to verify that \( S_{\vec m_n}(\z)S_{\vec k_n}(\z^*) = S_{\vec m_n-\vec k_n}(\z) \). Hence, it holds that
\[
1/(\gamma_n\gamma_n^*) = \big(\Theta_nS_{\vec m_n-\vec k_n-\vec\tau}\big)\big(\infty^{(0)}\big)\lim_{z\to\infty}\big(\widetilde\Theta_n\Phi)\big(z^{(1)}\big)z^{-g-1}.
\]
The claim \eqref{gammas-asymp} now follows from the boundedness of the vectors \( \vec m_n-\vec k_n-\vec\tau \) and therefore of the corresponding Szeg\H{o}-type functions, the continuity of the dependence of the theta functions on the divisors \( \mathcal D_n \) and \( \widetilde{\mathcal D}_n \), and the fact that the sets \( \mathfrak D \) and \( \widetilde{\mathfrak D} \) contain no divisors with involution-symmetric pairs (in which case the corresponding theta function would be identically zero), nor divisors containing \( \infty^{(0)} \) in the case of \( \mathfrak D \) (otherwise the theta function would vanish at \( \infty^{(0)} \)), nor divisors containing \( \infty^{(1)} \) in the case of \( \widetilde{\mathfrak D} \) (otherwise the theta function would have a pole of order strictly less than \( g \) at \( \infty^{(1)} \)).

\section{Asymptotics of the Approximants}
\label{sec:asymp}

For brevity, let us set
\[
\boldsymbol I = \left(\begin{matrix} 1 & 0 \\ 0 & 1 \end{matrix}\right) \qandq \sigma_3 = \left(\begin{matrix} 1 & 0 \\ 0 & -1 \end{matrix}\right).
\]
To prove Theorem~\ref{thm:main}, we follow by now classical approach of Fokas, Its, and Kitaev \cite{FIK91,FIK92} connecting orthogonal polynomials to matrix Riemann-Hilbert problems and then utilizing the non-linear steepest descent method of Deift and Zhou \cite{DZ93}. To deal with non-analytic densities, we use the idea of extensions with controlled \( \bar\partial \)-derivative introduced by Miller and McLaughlin \cite{McLM08} and adapted to the setting of Pad\'e approximants by Baratchart and the author \cite{BY10}.

\subsection{Riemann-Hilbert Approach}

Consider the following Riemann-Hilbert problem for \( 2\times2 \) matrix functions (\rhy):
\begin{itemize}
\label{rhy}
\item[(a)] \( \boldsymbol Y \) is analytic in \( \overline\C\setminus\Delta \) and \( \displaystyle \lim_{z\to\infty} \boldsymbol Y(z)z^{-n\sigma_3} = \boldsymbol I \);
\item[(b)] \( \boldsymbol Y \) has continuous traces on \( \Delta^\circ \) that satisfy \( \displaystyle \boldsymbol Y_+ = \boldsymbol Y_- \left(\begin{matrix} 1 & \rho/(v_{2n}w_+) \smallskip \\ 0 & 1 \end{matrix}\right) \);
\item[(c)] \( \boldsymbol Y \) is bounded near those points in \( \Delta\setminus\Delta^\circ \) that do not belong to \( E \) and
\[
\boldsymbol Y(z) = \mathcal O\left(\begin{matrix} 1 & |z-e|^{-1/2} \smallskip \\ 1 & |z-e|^{-1/2}\end{matrix}\right)
\]
as \( D\ni z\to e \) near each \( e\in E \), where \( \Delta^\circ \) is the union of all the smooth points of \( \Delta \) (the collection of all the Jordan arcs in \( \Delta \) without their endpoints).
\end{itemize}

To connect \hyperref[rhy]{\rhy} to the polynomials \( q_n \), we also need to introduce near diagonal multi-point Pad\'e approximants
\[
[n+1/n-1;V_{2n}]_{f_\rho} =: \frac{\widetilde p_n}{\widetilde q_n}, \quad \widetilde R_n := \frac{\widetilde q_nf_\rho-\widetilde p_n}{v_{2n}},
\]
see Definition~\ref{def:pade}. Then the following lemma holds.

\begin{lemma}
\label{lem:rhy}
Assume that the polynomial \( q_n \) and the function \( \widetilde R_n \) are such that
\begin{equation}
\label{assumption}
\deg(q_n)=n \quad \text{and} \quad \widetilde R_n(z)\sim z^{-n} \qasq z\to\infty.
\end{equation}
Let \( k_n \) be a constant such that \( k_n\widetilde R_n(z)=z^{-n}[1+o(1)] \) near infinity. Then \hyperref[rhy]{\rhy} is solved by
\begin{equation}
\label{eq:y}
\boldsymbol Y = \left(\begin{matrix}
q_n & R_n \\
k_n \widetilde q_n & k_n\widetilde R_n
\end{matrix}\right).
\end{equation}
Conversely, if \hyperref[rhy]{\rhy} is solvable, then its solution necessarily has the form \eqref{eq:y} and the polynomial \( q_n \) and the function \( \widetilde R_n \) satisfy \eqref{assumption}.
\end{lemma}
\begin{proof}
Let \( \boldsymbol Y \) be given by \eqref{eq:y}. The functions \( R_n,\widetilde R_n \) are clearly holomorphic outside of \( \Delta \). Since \( R_n(z) = \mathcal O\big(z^{-n-1}\big) \), \( \deg(\widetilde q_n) = n-1 \), and we assume \eqref{assumption}, \hyperref[rhy]{\rhy}(a) is immediate. It follows from Sokhotski-Plemelj formulae \cite[Section~4.2]{Gakhov} that
\[
f_{\rho+}(s) - f_{\rho-}(s) = \rho(s)/w_+(s), \quad s\in\Delta^\circ.
\]
Therefore, \hyperref[rhy]{\rhy}(b) is an easy consequence of \eqref{Rn}. Furthermore, both functions \( R_n,\widetilde R_n \) behave like \( \mathcal O\big(|z-e|)^{-1/2}\big) \) near \( e\in E \) by \cite[Section 8.4]{Gakhov} and near those \( e\in\Delta\setminus\Delta^\circ \) that are not in \( E \) it holds that
\[
f_\rho(z) = \frac1{2\pi\mathrm i}\left(\sum_{\Delta_j}\lim_{\Delta_j\ni s\to e}\frac{\rho(s)}{w_+(s)}\right) \log|z-e| + \mathcal O(1),
\]
where the sum is taken over all the open Jordan arcs \( \Delta_j\subseteq\Delta^\circ \) incident with \( e \). Since \( \rho \) is continuous at \( e \) which is not a branch point of \( w \), the sum in parenthesis is equal to zero. Hence, the functions \( R_n,\widetilde R_n \) are indeed bounded near such \( e \) and \hyperref[rhy]{\rhy}(c) does hold for \( \boldsymbol Y \) given by \eqref{eq:y}.

Conversely, let \( \boldsymbol Y \) be a solution of \hyperref[rhy]{\rhy}. It is necessarily unique. Indeed, \( \det(\boldsymbol Y) \) is a holomorphic function in \( \overline\C\setminus (\Delta\setminus\Delta^\circ) \) and \( \det(\boldsymbol Y)(\infty)=1 \). Since it has at most square root singularity at points of \( \Delta\setminus\Delta^\circ \), those singularities are in fact removable and therefore \( \det(\boldsymbol Y) \) is a bounded entire function. That is, \( \det(\boldsymbol Y)\equiv1 \) as follows from the normalization at infinity. Hence, if \( \widetilde{\boldsymbol Y} \) is another solution, \( \widetilde{\boldsymbol Y}\boldsymbol Y^{-1} \) is an entire matrix-function which is equal to \( \boldsymbol I \) at infinity, i.e., \( \widetilde{\boldsymbol Y} = \boldsymbol Y \). 

Now, we see from \hyperref[rhy]{\rhy}(a,b) that \( [\boldsymbol Y]_{11} \)  is a monic polynomial of degree \( n \). We also see that \( [\boldsymbol Y]_{12} - R_n \) has no jump on \( \Delta^\circ \) and can have at most square root singularities at \( e\in E \). Thus, \( [\boldsymbol Y]_{12} = ([\boldsymbol Y]_{11}f_\rho - p)/v_{2n} \) for some polynomial \( p \). Since \( [\boldsymbol Y]_{12} \) is holomorphic in \( D \) and vanishes at infinity with order at least \( n+1 \),  \( p,[\boldsymbol Y]_{11} \) are solutions of the linear system \eqref{Rn}. The uniqueness yields that \( [\boldsymbol Y]_{11} = q_n \) and \( p = p_n \). The second row of \( \boldsymbol Y \) can be analyzed analogously.
\end{proof}

\subsection{Riemann-Hilbert-\( \bar\partial \) Problem}

The next step is based on separating the jump in \hyperref[rhy]{\rhy}(b) into two and moving one of them away from \( \Delta \). This will require  extending \( \rho \) from \( \Delta \) into the complex plane. If \( \rho \) is holomorphic in some neighborhood of \( \Delta \), then this is the extension we shall use. Otherwise our construction is based on the following specialization of \cite[Theorem~1.5.2.3]{Grisvard}.

\begin{theorem}
\label{thm:sobolev}
Let \( L_1 \) and \( L_2 \) be two disjoint open analytic arcs with common endpoints that meet at non-zero angles there. Let \( g_i \), \( i\in\{1,2\} \), be a function in \( W_p^{1-1/p}(L_i) \), \( p>2 \) (replace \( \Delta \) with \( L_i \) in \eqref{sobolev}). If \( g_1 \) and \( g_2 \) have the same values at the endpoints of the arcs, then there exists a function \( G \in W_p^1(O) \) such that its boundary values on \( L_i \) are equal to \( g_i \), where \( O \) is the bounded domain delimited by \( L_1 \) and \( L_2 \), and \( W_p^1(O) \) is the subspace of \( L^p(O) \) consisting of functions whose weak partial derivatives are also in \( L^p(O) \). The construction of the function \( G \) is independent of \( p \).
\end{theorem}

Let \( \log\rho \) be a continuous determination of the logarithm of \( \rho \) on \( \Delta \). Further, let \( g \) be the polynomial of minimal degree interpolating \( \log\rho \) the points of \( \Delta\setminus\Delta^\circ \). For each subarc \( \Delta_j \) of \( \Delta \) select two analytic subarcs \( \Delta_{j+},\Delta_{j_-} \) that have the same endpoints as \( \Delta_j \) and lie to the left and right of  \( \Delta_j \) (according to the chosen orientation), see Figure~\ref{fig:3}.
\begin{figure}[!ht]
\centering
\includegraphics[scale=1]{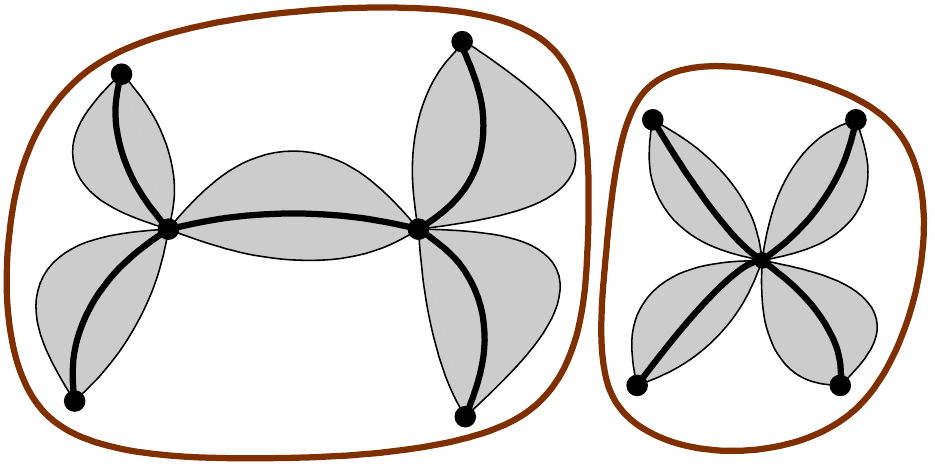}
\begin{picture}(0,0)
\put(-270,100){\( \Gamma \)}
\put(-240,48){\( \Omega_{1+} \)}
\put(-260,90){\( \Omega_{2+} \)}
\put(-195,80){\( \Omega_{3+} \)}
\put(-185,63){\( \Omega_{3-} \)}
\put(-130,90){\( \Omega_{4+} \)}
\put(-133,55){\( \Omega_{5-} \)}
\put(-93,77){\( \Omega_{6-} \)}
\put(-30,77){\( \Omega_{7+} \)}
\put(-32,40){\( \Omega_{8-} \)}
\put(-93,45){\( \Omega_{9+} \)}
\end{picture}
\caption{The system of curves \( \Gamma \) and some of the extension domains \( \Omega_{j\pm} \) (the labeling of the arcs \( \Delta_j \) is as on Figure~\ref{fig:1}). }
\label{fig:3}
\end{figure}
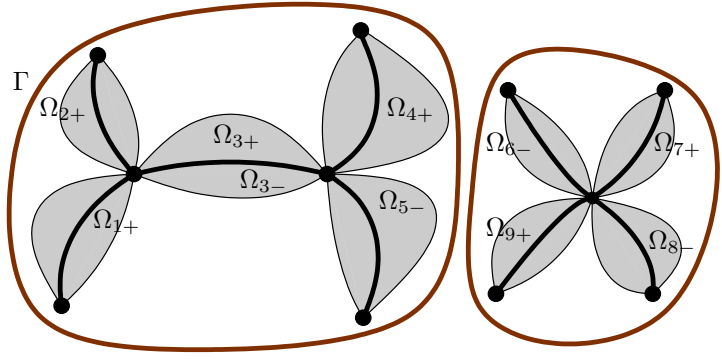
Assume in addition that all the arcs \( \Delta_j,\Delta_{j+},\Delta_{j_-} \) are disjoint and form definite angles at the common endpoints. Denote by \( \Omega_{j\pm} \) the domain delimited by \( \Delta_j \) and \( \Delta_{j\pm} \). Then, according to Theorem~\ref{thm:sobolev}, there exists a function \( G \) such that
\[
G_{|\Delta_j} = \log\rho, \quad G_{|\Delta_{j\pm}} = g, \qandq \bar\partial G \in L^p(\Omega_{j\pm}),
\]
for every \( j \). Then we can extend the function \( \rho \) from \( \Delta \) by
\begin{equation}
\label{extension}
\rho(z) := \left\{
\begin{array}{ll}
e^{G(z)}, & z\in\overline\Omega_{j\pm}, \medskip \\
e^{g(z)}, & \text{otherwise}.
\end{array}
\right.
\end{equation}
Observe further that in this case
\begin{equation}
\label{partial}
\bar\partial(1/\rho) := \left\{
\begin{array}{ll}
-\bar\partial G/\rho, & \text{in }\overline\Omega_{j\pm}, \medskip \\
0, & \text{otherwise}.
\end{array}
\right.
\end{equation}

Now, \( \Gamma \) be a union of simple Jordan curves each encompassing one connected component of \( \Delta \) and chosen so \( \rho \) is holomorphic across \( \Gamma \) if \( \rho \) is a holomorphic function and so that \( \overline\Omega_{j\pm} \) are contained in the interior of \( \Gamma \), say \( \Omega \), see Figure~\ref{fig:3}. Using extension \eqref{extension} when necessary, set
\begin{equation}
\label{eq:x}
\boldsymbol X:= \left\{
\begin{array}{ll}
\boldsymbol Y \left(\begin{matrix} 1 & 0 \smallskip \\ -v_{2n}w / \rho & 1 \end{matrix}\right), & \mbox{in} \quad \Omega, \medskip \\
\boldsymbol Y, & \mbox{in} \quad \C\setminus\overline\Omega.
\end{array}
\right.
\end{equation}
It is trivial to verify that \( \boldsymbol X \) solves the following Riemann-Hilbert-\( \bar\partial \) problem (\rhx):
\begin{itemize}
\label{rhx}
\item[(a)] \( \boldsymbol X \) is continuous in \( \C\setminus(\Delta\cup\Gamma) \) and \( \lim_{z\to\infty} \boldsymbol X(z)z^{-n\sigma_3} = \boldsymbol I \);
\item[(b)] \( \boldsymbol X \) has continuous traces on \( \Delta^\circ\cup\Gamma \) that satisfy
\[
\boldsymbol X_+=\boldsymbol X_- \left\{
\begin{array}{rl}
\displaystyle \left(\begin{matrix} 0 & \rho/(v_{2n}w_+) \smallskip \\ -v_{2n}w_+/\rho & 0 \end{matrix}\right) & \text{on} \quad \Delta^\circ,  \medskip \\
\displaystyle \left(\begin{matrix} 1 & 0 \smallskip \\ v_{2n}w/\rho & 1 \end{matrix}\right) & \text{on} \quad \Gamma;
\end{array}
\right.
\]
\item[(c)] \( \boldsymbol X \) has the behavior near \( e\in \Delta\setminus\Delta^\circ \) described by \hyperref[rhy]{\rhy}(c);
\item[(d)] \( \boldsymbol X \) deviates from an analytic matrix function according to
\[
\bar\partial \boldsymbol X = \boldsymbol X \left(\begin{matrix} 0 & 0 \\ v_{2n}w\bar\partial G/\rho & 0 \end{matrix}\right),
\]
where we extend \( \bar\partial G \) by zero outside of \( \overline\Omega_{j\pm} \), see \eqref{partial}.
\end{itemize}

One can readily verified that the following lemma holds.

\begin{lemma}
\label{lem:rhx}
\hyperref[rhx]{\rhx} is solvable if and only if \hyperref[rhy]{\rhy} is solvable. When solutions of \hyperref[rhx]{\rhx} and \hyperref[rhy]{\rhy} exist, they are unique and connected by \eqref{eq:x}.
\end{lemma}

\subsection{Analytic Approximation}

Below we would like to construct a matrix function that solves \rha:
\begin{itemize}
\label{rha}
\item[(a)] \( \boldsymbol A \) is continuous in \( \C\setminus(\Delta\cup\Gamma) \) and \( \lim_{z\to\infty} \boldsymbol A(z)z^{-n\sigma_3} = \boldsymbol I \);
\item[(b)] \( \boldsymbol A \) has continuous traces on \( \Delta^\circ\cup\Gamma \) that satisfy \hyperref[rhx]{\rhx}(b);
\item[(c)] \( \boldsymbol A \) has the behavior near \( e\in \Delta\setminus\Delta^\circ \) described by \hyperref[rhy]{\rhy}(c).
\end{itemize}

As we shall show later, the jumps of \( \boldsymbol A \) on \( \Gamma \) are asymptotically negligible. Hence, \( \boldsymbol A \) is asymptotically close to a matrix function solving \rhn:
\begin{itemize}
\label{rhn}
\item[(a)] \( \boldsymbol N \) is analytic in \( \C\setminus\Delta \) and \( \lim_{z\to\infty} \boldsymbol N(z)z^{-n\sigma_3} = \boldsymbol I \);
\item[(b)] \( \boldsymbol N \) has continuous traces on \( \Delta^\circ \) that satisfy
\[
\boldsymbol N_+=\boldsymbol N_- \left(\begin{matrix} 0 & \rho/(v_{2n}w_+) \smallskip \\ -v_{2n}w_+/\rho & 0 \end{matrix}\right);
\]
\item[(c)] \( \boldsymbol N \) has the behavior near \( e\in\Delta\setminus\Delta^\circ \) described by \hyperref[rhy]{\rhy}(c).
\end{itemize}

\begin{lemma}
\label{lem:rhn}
For all \( n\in \N_* \) the problem \hyperref[rhn]{\rhn} is solved by
\begin{equation}
\label{eq:n}
\boldsymbol N := \boldsymbol C \boldsymbol M, \quad \boldsymbol C:=\left(\begin{matrix} \gamma_n & 0 \smallskip \\ 0 & \gamma_n^* \end{matrix}\right) \qandq \boldsymbol M :=\left(\begin{matrix} \Psi_n & \Psi_n^*/w \smallskip \\ \Psi_n\Upsilon_n &  \Psi_n^*\Upsilon_n^*/w \end{matrix}\right),
\end{equation}
where the functions \( \Psi_n,\Psi_n^*,\Upsilon_n,\Upsilon_n^*\) are defined by \eqref{Psis} and the constants \( \gamma_n,\gamma_n^* \) by \eqref{gammas}. Moreover, \( \det(\boldsymbol N) \equiv 1 \) in \( \overline \C \).
\end{lemma}
\begin{proof}
\hyperref[rhn]{\rhn}(a) follows immediately from the analyticity properties of the functions \( \Psi_n,\Psi_n^*,\Upsilon_n,\Upsilon_n^*\) and the very way the constants \( \gamma_n,\gamma_n^* \) were defined. \hyperref[rhn]{\rhn}(b) can be easily checked by using \eqref{PsiBoundary}. Finally, \hyperref[rhn]{\rhn}(c) is the consequences of the boundedness of the traces of \( \Psi_n,\Psi_n^*,\Upsilon_n,\Upsilon_n^*\) on \( \Delta \) and the definition of \( w \). The identity \( \det(\boldsymbol N) \equiv 1 \) can be shown as in the proof of Lemma~\ref{lem:rhy}.
\end{proof}

To deal with the jump of \( \boldsymbol A \) on \( \Gamma \), we need a matrix function solving \rhr:
\begin{itemize}
\label{rhr}
\item[(a)] \( \boldsymbol Z \) is a holomorphic matrix function in \( \overline\C\setminus\Gamma \) and \( \boldsymbol Z(\infty)=\boldsymbol I \);
\item[(b)] \( \boldsymbol Z \) has continuous traces on \( \Gamma \) that satisfy
\[
\boldsymbol Z_+  = \boldsymbol Z_- \boldsymbol M \left(\begin{matrix} 1 & 0 \\ v_{2n}w/\rho & 1 \end{matrix}\right) \boldsymbol M^{-1}.
\]
\end{itemize}
Then the following lemma takes place.
\begin{lemma}
\label{lem:rhr}
The solution of \hyperref[rhr]{\rhr} exists for all \( n\in\N_* \) large enough and satisfies
\begin{equation}
\label{eq:r}
\boldsymbol Z=\boldsymbol I+o(1)
\end{equation}
uniformly in \( \overline\C \).
\end{lemma}
\begin{proof}
Since \( \det(\boldsymbol N)\equiv1 \) and therefore \( \det(\boldsymbol M)\equiv1/(\gamma_n\gamma_n^*) \), the jump matrix for \( \boldsymbol Z \) is equal to
\begin{equation}
\label{smalljump}
\boldsymbol I + \gamma_n\gamma_n^*\frac{v_{2n}}{\rho w}\big(\Psi_n^*\big)^2 \left(\begin{matrix} \Upsilon_n^* & -1 \smallskip \\ \big(\Upsilon_n^*\big)^2 &  -\Upsilon_n^* \end{matrix}\right) = \boldsymbol I + o(1),
\end{equation}
where the last equality follows from \eqref{gammas-asymp}, \eqref{maxPsi}, and Definition~\ref{def:symmetry}(ii). It was shown in \cite[Corollary~7.108]{Deift} that \eqref{smalljump} implies solvability \hyperref[rhr]{\rhr} for all \( n\in\N_* \) large enough as well as estimate \eqref{eq:r}.
\end{proof}

The verification of the following lemma is rather trivial.

\begin{lemma}
\label{lem:rha}
Let \( \boldsymbol N=\boldsymbol{CM} \) be  the solution of \hyperref[rhn]{\rhn} granted by Lemma~\ref{lem:rhn} and \( \boldsymbol Z \) be the solution of \hyperref[rhr]{\rhr} granted by Lemma~\ref{lem:rhr}. Then it can be easily checked that \( \boldsymbol A:=\boldsymbol{CZM} \) solves \hyperref[rha]{\rha}.
\end{lemma}

\subsection{\( \bar\partial \) Problem}

In this section we are looking for a solution of the following \( \bar\partial \)-problem (\pbd):
\begin{itemize}
\label{pbd}
\item[(a)] \( \boldsymbol D \) is a continuous matrix function in \( \overline\C \) and \( \boldsymbol D(\infty)=\boldsymbol I \);
\item[(b)] \( \boldsymbol D \) deviates from an analytic matrix function according to \( \bar\partial \boldsymbol D = \boldsymbol {DW} \), where 
\[
\boldsymbol W := \boldsymbol {ZM} \left(\begin{matrix} 0 & 0 \\ v_{2n}w\bar\partial G/\rho & 0 \end{matrix}\right)\boldsymbol M^{-1}\boldsymbol Z^{-1},
\]
\( \boldsymbol Z \) is the solution of \hyperref[rhr]{\rhr}, and \( \boldsymbol M \) is defined in \eqref{eq:n}.
\end{itemize}

Then the following lemma holds.

\begin{lemma}
\label{lem:pbd}
The solution of \hyperref[pbd]{\pbd} exists for all \( n\in\N_* \) large enough and satisfies
\begin{equation}
\label{eq:d}
\boldsymbol D=\boldsymbol I+o(1)
\end{equation}
locally uniformly in \( D \).
\end{lemma}
\begin{proof}
Let \( O \) be an open set and \( \phi\in L^p(O) \). Define the Cauchy area integral of \( \phi \)  by
\[
\mathcal K\phi(z) := \frac1{2\pi\mathrm i}\iint_O \frac{\phi(s)}{s-z}\dd s \wedge \dd\bar s, \quad z\in O.
\]
It is known that \( \bar\partial \mathcal K\phi = \phi \), see \cite[Section~4.9]{AstalaIwaniecMartin}. Moreover, when \( p>2 \), \( \mathcal K \) is a bounded operator from \( L^p(O) \) into \( C^{1-2/p}(\overline O) \), the space of Lipschitz continuous functions in \( \overline O \) with exponent \( 1-2/p \), see \cite[Theorem~4.3.13]{AstalaIwaniecMartin}.  In fact, since we clearly can take \( z\not\in O \) in the definition of \(\mathcal K\phi \), it is well defined in the entire extended complex plane, is holomorphic outside of \( \overline O \), and is vanishing at infinity. Furthermore, since an extension of \( \phi \) by zero to any open set containing \( O \) is still in \( L^p \) of that set, \( \mathcal K\phi \) is necessarily H\"older continuous across \( \partial O \).

Let now \( O \) be such that \( \overline\Omega_{j\pm}\subset O \) and \( \overline O \subset \Omega \). Assume that there exists a bounded matrix function \( \boldsymbol D \) such that
\begin{equation}
\label{in-eq}
\boldsymbol I = (\mathcal I - \mathcal K_{\boldsymbol W})\boldsymbol D,
\end{equation}
where \( \mathcal I \) is the identity operator and \( \mathcal K_{\boldsymbol W}\boldsymbol D := \mathcal K(\boldsymbol{DW})\). Then properties of the Cauchy integral operator imply that this \( \boldsymbol D \) solves \hyperref[pbd]{\pbd}. 

As far as the solvability of \eqref{in-eq} is concerned, if \( \| \mathcal K_{\boldsymbol W} \|<1 \), where we consider \( \mathcal K_{\boldsymbol W} \) as an operator from the space of bounded matrix functions into itself, then \( (\mathcal I - \mathcal K_{\boldsymbol W})^{-1} \) exists as a Neumann series and
\[
\boldsymbol D = \boldsymbol I + \mathcal O\left(\frac{\| \mathcal K_{\boldsymbol W} \|}{1-\| \mathcal K_{\boldsymbol W} \|}\right).
\]
Moreover, \( \mathcal D \) satisfies \eqref{eq:d} if \( \| \mathcal K_{\boldsymbol W} \|=o(1) \). Hence, it only remains to prove this estimate. It holds that
\[
\| \mathcal K_{\boldsymbol W} \| \leq C\max_{i,j}\max_{z\in\overline O}\left\|\frac{[\boldsymbol W]_{ij}}{z-\cdot}\right\|_1
\]
for some absolute constants \( C \), where \( \|\cdot\|_q \) is the \( L^q(O) \)-norm. By the very definition, it holds that
\[
\boldsymbol W = \gamma_n\gamma_n^*\bar\partial G\frac{v_{2n}}{\rho w}\big(\Psi_n^*\big)^2 \boldsymbol Z\left(\begin{matrix} \Upsilon_n^* & -1 \smallskip \\ \big(\Upsilon_n^*\big)^2 &  -\Upsilon_n^* \end{matrix}\right)\boldsymbol Z^{-1}.
\]
Using \eqref{second-est}, \eqref{second-est2}, \eqref{gammas-asymp}, and \eqref{eq:r}, we get that
\[
\| \mathcal K_{\boldsymbol W} \| \leq C_1(\N_*)\max_{z\in\overline O}\left\|\frac{\Phi_n^*\bar\partial G/w}{z-\cdot}\right\|_1 \leq C(\N_*,O) \|\Phi_n\|_q
\]
for any \( q\in\big(\frac{4p}{p-4},\infty\big) \), where \( \Phi_n^*(z)=\Phi_n\big(z^{(1)}\big)\) and the second inequality follows by repeated application of H\"older inequality (recall that \( \bar\partial G\in L^p(O) \) and \( p>4 \)). Let \( \Gamma_n \subset O \) be a union of simple Jordan curves each encompassing one connected component of \( \Delta \). Denote by \( O_n \) the union of the bounded components of the complement of \( \Gamma_n \). Assume further that \(|O_n|\to0 \) as \( n\to\infty \), where \( |O_n| \) is the planar Lebesgue measure of \( O_n \). Then
\[
\|\Phi_n^*\|_q^q \leq \|\Phi_n^* \|_\Delta^q|O_n| + \|\Phi_n^* \|_{\Gamma_n}^q|O\setminus O_n| = o(1)
\]
by Definition~\ref{def:symmetry}(ii) as desired, where \( \|\cdot\|_K \) is the supremum norm of \( K \).
\end{proof}

\subsection{Asymptotics}

Given \( \boldsymbol A=\boldsymbol{CZM} \), constructed in Lemma~\ref{lem:rha}, and \(\boldsymbol D \), whose existence is guaranteed by Lemma~\ref{lem:pbd}, one can easily check that \( \boldsymbol X=\boldsymbol{CDZM} \) solves \hyperref[rhx]{\rhx}. It follows from Lemma~\ref{lem:rhx} that \hyperref[rhy]{\rhy} is solved by inverting \eqref{eq:x}.  Given any closed set \( K\subset\overline\C\setminus\Delta \), choose \( \Omega \) and \( \Omega_{j\pm} \) so that \( K\subset \overline\C\setminus\overline\Omega \). Then \( \boldsymbol Y = \boldsymbol X \).  Write
\[
\boldsymbol {DZ} = \left(\begin{matrix} 1+\varepsilon_{n1} & \varepsilon_{n2} \\ \varepsilon_{n3} & 1+\varepsilon_{n4} \end{matrix}\right),
\]
where \( |\varepsilon_{nk}|=o(1) \) locally uniformly in \( D \) by \eqref{eq:r} and \eqref{eq:d} and \( \varepsilon_{nk}(\infty)=0 \) as \( \boldsymbol{DZ}(\infty)=\boldsymbol I \). Then
\[
[\boldsymbol{Y}]_{1i} = \big(1+\varepsilon_{n1}\big)\gamma_n[\boldsymbol{M}]_{1i} + \varepsilon_{n2}\gamma_n[\boldsymbol{M}]_{2i}, \quad i\in\{1,2\},
\]
on \( K \). The claim of Theorem~\ref{thm:main} now follows from Lemma~\ref{lem:rhy} and the definition of \( \boldsymbol M \) in~\eqref{eq:n}.





\end{document}